\documentclass[11pt]{amsart}
\usepackage{geometry}                
\usepackage{graphicx}
\usepackage{amssymb}
\usepackage{epstopdf}
\newtheorem{theorem}{Theorem}[section]
\newtheorem{lemma}{Lemma}[section]
\newtheorem{corollary}{Corollary}[section]

\newtheorem{definition}{Definition}[section]

\numberwithin{equation}{section}

\def\Z{\mathbb Z}

\def\R{\mathbb R}

\def\d{\partial}
\def\a{\alpha}
\def\b{\beta}
\def\g{\gamma}

\def\e{\epsilon}
\def\D{\Delta}

\title{Flows in Flatland: A Romance of Few Dimensions }
\author{Gabriel Katz}

\address{5 Bridle Path Circle, Framingham, MA 01701, USA}
\email{gabkatz@gmail.com}

\begin{document}
\maketitle

\begin{abstract} In this paper, we present our general results about traversing flows on manifolds with boundary in the context of the flows on surfaces with boundary. We take advantage of the relative simplicity of $2D$-worlds to explain and popularize our approach to the Morse theory on smooth manifolds with  boundary, in which the boundary effects take the central stage.   
\end{abstract} 

\section{Introduction}

This paper is about the gradient flows on compact surfaces, thus the reference to Abbott's \emph{Flatland} \cite{Ab} in the title. The paper is an informal introduction into the philosophy and some key results from \cite{K} -\cite{K6}, as they manifest themselves in $2D$. 

The remarkable convergence of topological, geometrical, and analytical approaches to the study of closed surfaces is widely recognized by the practitioners for more than a century.  We will exhibit a similar convergence of different investigative approaches to vector flows on surfaces \emph{with boundary}. 

We will take advantage of the relative simplicity of $2D$  flows to illustrate and popularize the main ideas of our recent research of \emph{traversally generic} flows on manifolds with boundary. 
When the results are specific to the dimension two, their validation will  be presented in detail.  The multidimensional arguments that resist significant simplifications in $2D$ will be described and explained in general terms.  

Throughout the investigation, we focus on the interactions of gradient flows with the boundary, rather than on the critical points of Morse functions. So, in our approach to the Morse Theory, the boundary effects rule.

\section{On Morse Theory on surfaces with boundary and  beyond}

\emph{Morse Theory},  the classical book of John W. Milnor \cite{Mi}, starts with the canonical picture of a Morse function $f: T^2 \to \R$ on a 2-dimensional torus $T^2$ (see Fig. 1). It is portrayed as the height function $f$ on the torus $T^2$ residing in the space $\R^3$.  The height $f$ has four \emph{critical points}: $a, b, c$, and $d$ so that $$f(a) > f(b) > f(c) > f(d).$$ A point $z$  is called critical if the differential $df$ of $f$ vanishes at $z$. In the vicinity of each critical point $z$, $T^2$ admits a pair of local coordinate functions, say $x$ and $y$, so that locally the function $f$ acquires the form $$f(x, y) = f(0, 0) \pm x^2 \pm y^2,$$ where the signs may form four possible combinations.  

\begin{figure}[ht]\label{fig1.1}
\centerline{\includegraphics[height=2in,width=2.8in]{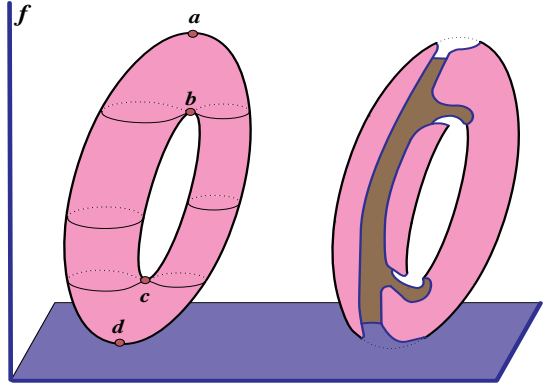}}
\bigskip
\caption{\small{A Morse function $f$ on a 2-dimensional torus $T^2$ and its non-singular restriction to the complement $X$ of a disk $D^2 \subset T^2$. Note the curved geometry of the boundary loop $\d X$ which ``remembers" the nature of $f$-critical points $a, b, c, d$.}}
\end{figure}

We call a  vector field $v$, tangent to $T^2$, \emph{gradient-like} if $df(v) > 0$ everywhere outside of the set $Cr(f)$ of critical points. 

If  the torus is ``slightly slanted" with respect to the vertical coordinate $f$ in $\R^3$, then the following picture emerges. The majority of downward trajectories of the $f$-gradient flow $\{\Phi_t\}_{t \in \R}$ that emanate from $a$, asymptotically reach $d$.  There are two trajectories that asymptotically link $a$ with $b$, and two trajectories that link $a$ with $c$. No (unbroken) trajectory asymptotically  connects $b$ to $c$. 

Perhaps, a more transparent  depiction of the gradient flow $\{\Phi_t\}_{t \in \R}$ is given in Fig. 2, where the torus is shown in terms of its fundamental domain, the square. To form $T^2$, the opposite sides of the square are identified in pairs. \smallskip

The Morse Theory is concerned with the sets of \emph{constant level} $\{f^{-1}(\a)\}_{\a \in \R}$ and the \emph{below constant level} sets $\{f^{-1}((-\infty, \a))\}_{\a \in \R}$. The main observation is that the topology of these sets is changing in an essential way only when the rising $\a$ crosses the critical values $$Cr(f) = \{f(a), f(b), f(c), f(d)\}.$$ Each such ``critical crossing" results in an \emph{elementary surgery} on the set $\{f^{-1}((-\infty, \a))\}_{\a \in \R}$, where $\a$ is just below a critical value $\a_\star \in Cr(f)$. For a small $\e > 0$, an elementary surgery $$f^{-1}((-\infty, \a_\star -\e)) \Rightarrow f^{-1}((-\infty, \a_\star + \e))$$ attaches the \emph{handle} $ f^{-1}((\a_\star -\e, \a_\star + \e))$ to the set $f^{-1}((-\infty, \a - \e))$. Eventually, when $\a$ rises above $f(a)$, the entire topology of torus $T^2$ is captured  by a sequence of these elementary surgeries. \smallskip

From a different angle,  the knowledge of how the critical points $a, b, c, d$ interact via the trajectories of the $\Phi_t$-flow is also sufficient for reconstructing the surface $T^2$ as Fig. 2 suggests (see \cite{C}).\smallskip

\begin{figure}[ht]\label{fig1.2}
\centerline{\includegraphics[height=1.8in,width=3in]{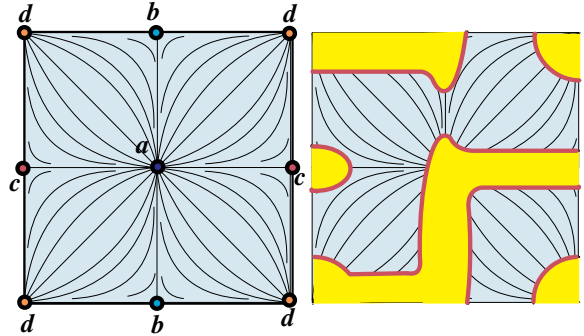}}
\bigskip
\caption{\small{The gradient flow of the Morse function $f: T^2 \to \R$ from Fig. 1 and its restriction to the complement of a disk in the torus.}}
\end{figure}

Note that, in the vicinity of each critical point, the gradient flow exhibits \emph{discontinuity}: small changes in the initial position of a point $z$, residing in the vicinity of a critical point, result in significant differences in the position of $\Phi_t(z)$ for big positive/small negative values of $t$ (see Fig. 3). In fact, this discontinuity of the gradient flow, expressed in terms of the \emph{stable} and \emph{unstable manifolds} of critical points (see \cite{Mi}),  captures the topology of the surface (as the left diagram in Fig. 2 suggests)! 

\begin{figure}[ht]\label{fig1.3}
\centerline{\includegraphics[height=1.3in,width=2.2in]{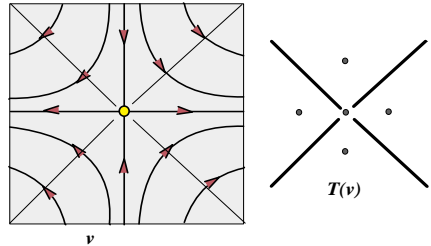}}
\bigskip
\caption{\small{A gradient flow $v$ in the vicinity of a singular point and a very schematic picture of its (nonseparable) trajectory space $\mathcal T(v)$.}}
\end{figure}

As a result of gradient flow discontinuity, the \emph{space of trajectories} $\mathcal T(v)$ is pathological (non-separable). The space $\mathcal T(v)$ is constructed by declaring equivalent any two points that reside on the same trajectory. 
\smallskip

When a compact connected surface $X$ has a nonempty boundary $\d X$, traditionally, the Morse function $f: X \to \R$ is assumed to be constant on $\d X$ and its gradient flow interacts with the boundary in constrained way. Then the \emph{relative topology} of the pair $(X, \d X)$ can be captured in the ways analogous to the previous description of the Morse Theory on torus. In fact, the Morse Theory on manifolds with boundary can be viewed as a very special instance of the Morse Theory on stratified spaces (the two strata  $\d X$ and $X$ form the stratification). The latter was developed by Goresky and MacPherson in [GM] -[GM2].
\smallskip

In this paper, we propose a different philosophy for the Morse Theory on compact surfaces/manifolds  $X$ \emph{with boundary}.  To formulate it, let us revisit our favorite closed surface, the torus. By deleting from $T^2$ small disks, centered on the points of the critical set $Cr(f)$, we manufacture a surface $X$ whose boundary is a disjoint union of four circles. Evidently, $f: X \to \R$ has no critical points at all. Still it has a nontrivial topology! Can this topology be reconstructed from some data, provided by the critical point-free $f$ and its gradient-like field $v \neq 0$? An experienced reader would notice that the restriction $f|: \d X \to \R$ has critical points (maxima and minima), some of which interact  \emph{along} the boundary (with the help of a gradient-like field $v^\d$, tangent to $\d X$). However, it is quite clear that these interactions are not sufficient for a  reconstruction of  the topology of $X$! In fact, a reconstruction of the surface $X$ becomes possible if one introduces additional interactions between the points of $Cr(f|_{\d X})$ that occur ``through the bulk $X$" and are defined with the help of \emph{both} vector fields $v$ and $v^\d$. This observation has been explored by a number of authors, but it is not the world view that we are promoting here... 

To  dramatize further the  situation we are facing, let us place four small disks, centered on the critical points of $f: T^2 \to \R$,  into a single open disk $D^2$ and form $X = T^2 \setminus D^2$ (see Fig 2, the right diagram). Again, $f|: X \to \R$ has no critical points, the gradient field $v|_X \neq 0$, but its topology of $X$ is nontrivial. This time, the boundary $\d X$  of the punctured torus $X$ is just a single circle!  Let us keep this challenge in mind. \smallskip 

 Can one propose a ``Morse Theory" that is not centered on critical points? The answer is affirmative. It relies on the following observation. Typically, in the vicinity of $\d X$, the $v$-trajectories are interacting with the boundary in a number of very particular and stable ways: they are either transversal to $\d X$, or are tangent to it in a \emph{concave} or \emph{convex} fashion\footnote{It is possible to have a field $v$ for which some trajectories will be cubically tangent to the boundary, but the majority of vector fields $v$ avoid such cubic tangencies.} (see Fig. 5). So the boundary $X$ may be ``wiggly" with respect to the flow. We claim that this geometry of the $v$-flow in connection to the boundary $\d X$ is the crucial ingredient  for reconstructions of $X$ in terms of the flow (see Section 8, especially Theorem \ref{th7.1}). 

In the vicinity of a \emph{concave} tangency point,  the $v$-flow is \emph{discontinuous} in the same sense as the gradient flow is discontinuous in the vicinity of its critical point: in time, close initial points become distant.  In this context, the divergence of initially close points occurs due to very different \emph{travel times} available to them; unlike the infinite travel time for the gradient flows of the Morse theory on \emph{closed} surfaces, in the case of the non-singular gradient flows on surfaces with boundary, every point exits the surface in finite time. In particular, the surface is \emph{not} flow-invariant. And again, these discontinuities of the flow reflect the topology of the surface.  Let us clarify this point.

Fig. 4 
shows  a gradient flow $v$ on a surface $X \subset \R^2$, the disk with $4$ holes. The nonsingular function $f: X \to \R$ is the vertical coordinate in $\R^2$.  Each $v$-trajectory is either a closed segment, or a singleton. By collapsing each trajectory to a point, we create a quotient space $\mathcal T(v)$ of trajectories. Since the flow trajectories are closed segments or singletons, this time, the trajectory space $\mathcal T(v)$ is ``decent",  a finite graph with verticies  of valency $1$ or $3$ only. The verticies of valency $3$ correspond to the  points on $\d X$ where the boundary  is concave with respect to the flow, and the univalent verticies to the points on $\d X$ where the flow is convex. 

The obvious map $\Gamma: X \to \mathcal T(v)$ cellular. Moreover, because the fibers of $\Gamma$ are contractable, $\Gamma$ is a \emph{homotopy equivalence}.  In particular, the fundamental groups $\pi_1(X)$ and $\pi_1(\mathcal T(v))$ are isomorphic with the help of $\Gamma$. So the trajectory spaces of generic non-vanishing vector fields $v$ of the gradient type on connected surfaces $X$ with boundary deliver $1$-dimensional \emph{homotopy theoretical models} of  $X$. 

\begin{figure}[ht]\label{fig1.4}
\centerline{\includegraphics[height=2.3in,width=3in]{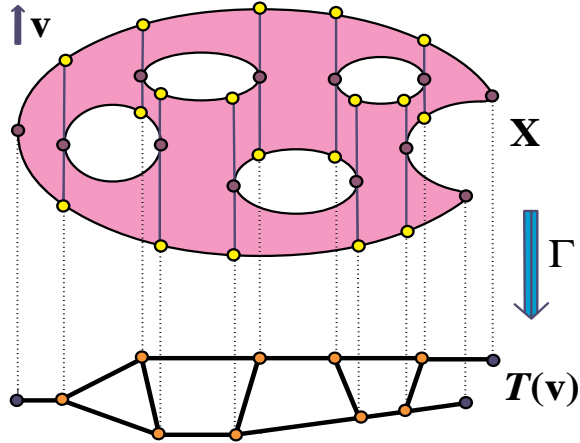}}
\bigskip
\caption{\small{The map $\Gamma: X \to \mathcal T(v)$ for a traversally generic (vertical) field $v$ on a disk with $4$ holes. The trajectory space $\mathcal T(v)$ is a graph whose verticies are of valecies $1$ and $3$.}}
\end{figure}

\section{Vector felds and Morse stratifications on surfaces}

Following \cite{Mo}, for any vector field $v$ on a compact surface $X$ with boundary such that $v|_{\d X} \neq 0$, we consider the closed locus $\d_1^+X(v)$, where the field is pointing inside $X$ and the closed locus $\d_1^-X(v)$, where it points outside. The intersection $$\d_2X(v) =_{\mathsf{def}} \d_1^+X(v)\cap \d_1^+X(v)$$ is the locus where $v$ is \emph{tangent} to the boundary $\d X$. Points $z \in \d_2X(v)$ come in two flavors: by definition, $z \in  \d^+_2X(v)$ when $v(z)$ points inside of the locus $\d_1^+X(v)$; otherwise  $z \in  \d^-_2X(v)$. To achieve some uniformity of notations, put $\d_0^+X =_{\mathsf{def}} X$ and $\d_1X =_{\mathsf{def}} \d X$. 

\begin{figure}[ht]\label{fig1.5}
\centerline{\includegraphics[height=1.2in,width=4in]{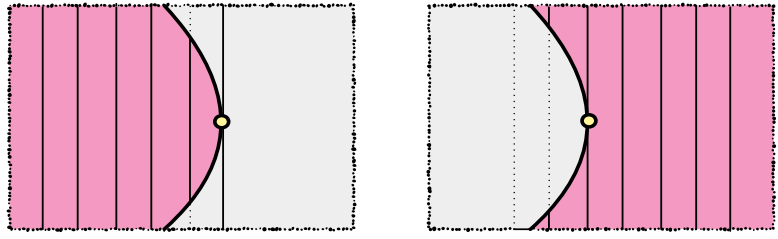}}
\bigskip
\caption{\small{A boundary generic field $v$ in the vicinity of a point from $\d_2^-X(v)$ (on the left) and in the vicinity of a point from $\d_2^+X(v)$ (on the right).}}
\end{figure}

\begin{definition}
We say that a vector field $v$ on a compact surface $X$ is \emph{boundary generic} if:
\begin{itemize}
\item $v|_{\d X}$, viewed as a section of the normal $1$-dimensional (quotient) bundle $$n_1 =_{\mathsf{def}} T(X)|_{\d X} \big/ T(\d X),$$ is transversal to its zero section, 
\item $v|_{\d_2 X(v)}$, viewed as a section of the normal $1$-dimensional bundle  $n_2 =_{\mathsf{def}} T(\d X)|_{\d X}$, is transversal to its zero section.  
\hfill $\diamondsuit$
\end{itemize} 
\end{definition}

\noindent In particular, for a boundary generic $v$, the loci $\d_1^\pm X(v)$ are finite unions of closed intervals and circles, residing in $\d X$; and the loci $\d_2^\pm X(v)$ are finite unions of points, residing in $\d X$ (see Fig. 4). \smallskip

We denote by $\mathcal V^\dagger(X)$ the space (in the $C^\infty$-topology) of all boundary generic fields on a compact surface $X$. \smallskip

Let $\chi(Z)$ denote the Euler number of a space $Z$. Recall that $\chi(Z)$ is the alternating sum of dimensions of the homology spaces $\{H_i(Z; \R)\}_i$. 

Since for a connected surface $X$ with boundary $H_2(X; \R) = 0$, we get
$$\chi(X) = 1 - \dim_\R(H_1(X; \R)).$$ For a closed connected surface, $$\chi(X) = 2 - \dim_\R(H_1(X; \R)).$$ 

Given a vector field $v$ with isolated zeros, we can associate an integer $\mathsf{ind}_x(v)$ with each zero $x$ of $v$. This integer is the \emph{degree} of the map which, crudely speaking, takes each point $z$ on a small circle $C_x$  with its center at $x$ to the unit vector $v(z)/\|v(z)\|$. Then we define $\mathsf{Ind}(v)$, the (global) index of $v$, as the sum $\sum_{\{x \in \text{ zeros of } v\}} \mathsf{ind}_x(v)$.\smallskip

The Morse formula \cite{Mo}, in the center of our investigation, computes the index $\mathsf{Ind}(v)$ of a given boundary generic vector field $v$ on a surface $X$ as the alternating sum of the Euler numbers of the Morse strata $\{\d_j^+X(v)\}_{0 \leq j \leq 2}$:
\begin{eqnarray}\label{eq1.1}
\mathsf{Ind}(v) = \chi(X) - \chi(\d_1^+X(v)) +  \chi(\d_2^+X(v)).
\end{eqnarray}
In the case of a connected surface $X$ with boundary, $\chi(X) = 1 - \dim_\R(H_1(X; \R))$, and this formula reduces to
$$\mathsf{Ind}(v) = 1 - \dim_\R(H_1(X; \R)) - \#\{ \mathsf{arcs}\; in \; \d_1^+X(v)\} +  \#\{\d_2^+X(v)\}$$
$$= 1 - \dim_\R(H_1(X; \R)) + \frac{1}{2}\big(\#\{\d_2^+X(v)\}- \#\{\d_2^-X(v)\}\big).$$ 
In particular,  if $v \neq 0$, then $\mathsf{Ind}(v) = 0$, and we get
\begin{eqnarray}\label{eq1.2}
 \frac{1}{2}\big(\#\{\d_2^+X(v)\}- \#\{\d_2^-X(v)\}\big) =  \dim_\R(H_1(X; \R)) - 1,
\end{eqnarray}
where the RHS of the equation is the topological invariant $|\chi(X)| = -\chi(X)$ of $X$. In contrast, the cardinality $\#\{\d_2^+X(v)\}$ depends on $v$.
 
 \begin{lemma}\label{lem3.1} Let a surface $X$ be formed by removing $k$ open disks from a closed surface $Y$, the sphere with $g$ handles. Then, for any boundary generic  field $v \neq 0$  on $X$,
 $$\#\{\d_2^+X(v)\}\, \geq \,  4g - 4 + 2k.$$
Moreover, $\#\{\d_2^+X(v)\} =  4g - 4 + 2k$ only when $\#\{\d_2^-X(v)\} = 0$. 
\end{lemma} 

\begin{proof} The Euler number is additive under gluing surfaces along their boundary components. Therefore, if $k$ disks are removed from $Y$, the sphere with $g$ handles, then  $\chi(X) = 2-2g - k$. Thus the Morse formulas  (\ref{eq1.1}) and (\ref{eq1.2}) imply  $$\#\{\d_2^+X(v)\} \geq 4g - 4 + 2k$$ for any   $v \neq 0$.  Moreover, $\#\{\d_2^+X(v)\} =  4g - 4 + 2k$ if and only if $\#\{\d_2^-X(v)\} = 0$, the main feature of the \emph{boundary concave} fields (see Definition \ref{def1.2}).
\end{proof}
In particular, for any non-vanishing boundary generic field $v$ on a torus with a single hole, $\#\{\d_2^+X(v)\} \geq 2$ (cf. Fig. 2).
\smallskip

Recall that an \emph{immersion} is a smooth map of manifolds, whose differential has the trivial kernel. 

Consider a smooth map $\a: X \to \R^2$, which is an immersion in the vicinity of $\d X$. Any such $\a$ gives rise to the Gauss map $G: \d X \to S^1$, defined by the formula $G(x) = \a_\ast(\tau_x)/\| \a_\ast(\tau_x)\|$, where $\tau_x$ is the tangent vector to $\d X$ at $x$.  The direction of $\tau_x$ is consistent with the preferred orientation of $\d X$, induced by the preferred orientation of $X$ .  

Let $\hat v \neq 0$ be a constant field on $\R^2$. Since the kernel of the differential of $D\a: TX \to T\R^2$ is trivial along $\d X$, the field $\hat v$ defines a vector field $\tilde v = \a^\ast(\hat v)$ on $X$ in the vicinity of $\d X$. The pull-back field $\tilde v$ extends to a vector field $v$ on $X$, possibly with zeros (see [G] for engaging discussions of vector field transfers and the Gauss-Bonnet Theorem).\smallskip

Then the degree of the Gauss map is given by a classical Hopf formula (\cite{H}) $$\deg(G) = \chi(X) - \mathsf{Ind}(v).$$

When  $\a: X \to \R^2$ is an immersion everywhere, the pull-back field $v = \a^\ast(\hat v) \neq 0$ everywhere. Thus $\mathsf{Ind}(v) = 0$, and,  for a  connected $X$ with $\d X \neq \emptyset$, we get
$$\deg(G) = \chi(X) =_{\mathsf{def}} 1 - \dim(H_1(X; \R)).$$ So, for an immersions $\a$, we get a new interpretation of formula (\ref{eq1.2}):
\begin{eqnarray}\label{eq1.3}
\deg(G) = \chi(X) = \frac{1}{2} \Big(\#\{\d_2^-X(v)\} - \#\{\d_2^+X(v)\}\Big).
\end{eqnarray}
This global-to-local formula has another classical geometrical interpretation. Let $\mathsf g = \a^\ast(\mathsf g_E)$ be the Riemannian metric on $X$, the pull-back of the Euclidean metric on $\R^2$. 
Let  $K_\nu$ denote the normal curvature of $\d X$ with respect to $\mathsf g$. Then $$\deg(G) = \frac{1}{2\pi}\int_{\d X} K_\nu\, d\mathsf g,$$ which leads to another pleasing global-to-local connection:
\[
\frac{1}{\pi}\int_{\d X} K_\nu\, d\mathsf g = \#\{\d_2^-X(v)\} - \#\{\d_2^+X(v)\}.
\]

In particular, for  a connected orientable surface $X$ of genus $g$ with a \emph{single} boundary component,  
\begin{eqnarray}\label{eq1.4}
\chi(X) = 1 - 2g  = \frac{1}{2}\big(\#\{\d_2^-X(v)\} - \#\{\d_2^+X(v)\}\big).
\end{eqnarray}

So the number of $v$-trajectories $\g$ in $X$ that are tangent to $\d X$, but are not singletons (they correspond  to points of $\d_2^+X(v)$), as a function of genus $g$,  grows at least as fast as $4g - 2$!
\smallskip

On the other hand, when $\d X$ is connected, by the Whitney index formula \cite{W}, the degree of the Gauss map $G: \d X \to S^1$ can be also calculated as $\mu + N^+ - N^-$, where $N^\pm$ denotes the number of positive/negative self-intersections of the curve $\a(\d X) \subset \R^2$, and $\mu = \pm 1$. Here is a brief description of the rule by which the self-intersections acquire polarities. Let $p \in \a(\d X)$ be a point where the coordinate function $y: \R^2 \to \R$ attends its minimum on the curve $\a(\d X)$. If the tangent vector $\tau_p$ at $p$, which defines the orientation of $\a(\d X)$, is $\d_x$, then we put $\mu = +1$; if $\tau_p = - \d_x$, then $\mu = -1$. Starting at $p$ and moving in the direction of $\tau_p$, we visit each self-intersection $a$ twice and in a particular order. The first visitation defines a tangent vector $\tau_1(a)$, the second visitation defines a tangent vector $\tau_2(a)$. When the ordered pair $(\tau_1(a), \tau_2(a))$ defines the clockwise orientation of the $xy$-plane, then we attach ``$-$" to $a$. Otherwise, the polarity of $a$ is ``$+$".
\smallskip

Therefore we get a somewhat mysterious connection between the self-intersections of $\d X$ under immersions $\a: X \to \R^2$ and the tangency patterns of the flows in $X$ that are the $\a$-pull-backs of  non-vanishing flows in the plane. 

\begin{theorem}\label{th1.2} Let $\hat v \neq 0$ be a vector field in the plane $\R^2$.  Let $X$ be a connected orientable surface with a connected boundary. Consider   an immersion $\a: X \to \R^2$ such that the loop $\a(\d X)$ has transversal self-intersections only. Assume that the pull-back $v = \a^\ast(\hat v)$ is a boundary generic field on $X$.   
Then 
$$\frac{1}{2} \Big(\#\{\d_2^+X(v)\} - \#\{\d_2^-X(v)\}\Big) = N^+ - N^- \pm 1 = 2g -1,$$
$$\frac{1}{2} \Big(\#\{\d_2^+X(v)\} - \#\{\d_2^-X(v)\}\Big) + 2\, \leq N^+ + N^-,$$ 
the latter inequality being sharp by an appropriate choice of $\a$.

\end{theorem}

\begin{proof} The first formula is the result of combining the Whitney formula for $\deg(G)$ with formulas (\ref{eq1.3}),  (\ref{eq1.4}).

By a theorem of Guth \cite{Gu}, for any immersion $\a: X \to \R^2$, the total number of self-intersections of the loop $\a(\d X)$ admits an estimate $$N^+ + N^- \geq \, 2g + 2.$$ Moreover, this lower bound is realized by an immersion $\a: X \to \R^2$! Therefore, by  formula  (\ref{eq1.4}), the Guth inequality is transformed into 
$$N^+ + N^- \geq  
2+ \frac{1}{2}\big(\#\{\d_2^+X(v)\} - \#\{\d_2^-X(v)\}\big).$$
Moreover, for some optimal immersion $\a$, 
$$N^+ + N^- =  2+ \frac{1}{2}\big(\#\{\d_2^+X(v)\} - \#\{\d_2^-X(v)\}\big)  = 2 - \frac{1}{2\pi}\int_{\d X} K_\nu\, d\mathsf g.$$

\end{proof}
\smallskip

When a surface $X$ is oriented and a field $v$ is boundary generic, then the points from $\d_2^+X(v)$ come in two \emph{new flavors}: ``$\oplus, \ominus$". 
By definition, a point $a \in \d_2^+X(v)$ has the polarity ``$\oplus$" if the orientation of $T_aX$ determined by the pair  $(\nu_a, v(a))$, where $\nu_a$ is the inner normal to $\d X$, agrees with the preferred orientation of $X$. Otherwise, the polarity of $a$ is defined to be ``$\ominus$". 

Thus, for each choice of orientation of $X$ (and hence of $\d X$) we get a partition $$\d_2^+X(v) = \d_2^{+, \oplus}X(v) \coprod \d_2^{+, \ominus}X(v).$$
Switching the orientation of $X$ switches the second polarities in the partition.

\section{Convexity,  concavity, and complexity of flows in 2D}

\begin{definition}\label{def1.2} We say that a boundary generic vector field $v$ is \emph{boundary convex} if $\d_2^+X(v) = \emptyset$. We say that a boundary generic $v$ is \emph{boundary concave} if $\d_2^-X(v) = \emptyset$ (see Fig 5). \hfill $\diamondsuit$
\end{definition}

The existence of a boundary convex field puts severe restrictions on the topology of the surface.

\begin{lemma}\label{lem1.2} If a compact connected surface $X$ with boundary $\d X \neq \emptyset$ admits a boundary convex gradient-like vector field $v \neq 0$, then $X$ is either a disk $D^2$, or an annulus $A^2$. 
\end{lemma}

\begin{proof} The convexity of the field $v$ implies that $X$ admits a $(-v)$-directed continuous retraction on the locus $\d_1^+X(v)$. Since $X$ is connected, it follows that $\d_1^+X(v)$ is connected as well. Thus, $\d_1^+X(v)$ is either a circle, or a segment. In the first case, $X$ is diffeomorphic to an annulus $S^1 \times [0, 1]$; in the second case,  $X$ is diffeomorphic to a disk $D^2$.
\end{proof}

The same phenomenon occurs  in any dimension: if a compact connected smooth $(n+1)$-manifold $X$ with a \emph{connected} boundary  admits a \emph{boundary convex} gradient-like vector field $v \neq 0$, then $H_n(X; \Z) = 0$ (\cite{K1}). In other words, $H_n(X; \Z) \neq 0$ is a topological obstruction to the existence of a boundary convex non-vanishing gradient field on $X$.
\smallskip

In contrast, the boundary concave non-vanishing gradient fields  are plentiful.  For example, consider a radial vector field $v$ on an annulus $A^2$. Delete from $A^2$ any number of convex disks and restrict $v$ to the resulting $2$-disk with holes. The convexity of the disks that we have removed implies that any disk with holes admits  a boundary concave gradient-like vector $v \neq 0$. 

Many other surfaces admit such concave fields as well. For example, consider a Morse function $f: Y \to \R$ on a closed surface $Y$ and its gradient field $v$. Then removing small convex (in the local Morse coordinates) balls, centered on the critical points, from $Y$, produces a \emph{boundary concave} non-vanishing gradient field on $X$. In particular if $Y$ is a sphere with $g$ handles, then one can find a Morse function with $2g + 2$ critical points (see Fig. 1). So the surface $X$, obtained from $Y$ by removing $2g + 2$ balls, admits a concave gradient-like field $v \neq 0$.

 In fact, by Theorem \ref{th6.2}, any connected orientable surface with boundary, but the disk, admits a boundary concave non-vanishing gradient field! \smallskip

We view the integer $c^+(v) =_{\mathsf{def}} \#(\d_2^+X(v))$ as a \emph{measure of complexity} of the $v$-flow, subject to the condition $\mathsf{Ind}(v) = 0$ or, alternatively, subject to the condition $v \neq 0$ .

We define the \emph{complexity} of a compact connected surface $X$ with boundary as the minimum 
$$c^+(X) = \min_{v \neq 0} \{c^+_2(v)\},$$
where $v$ runs over all non-vanishing  boundary generic fields on $X$.  

By varying $v$ within different spaces of fields, one may consider a variety of such minima;  non-vanishing fields and non-vanishing gradient-like fields are the two most important cases. So we introduce  the \emph{gradient complexity} $$gc^+(X) =_{\mathsf{def}} \min_{v \neq 0\; \text{of the gradient type}} \{c^+_2(v)\},$$
where $v$ runs over all non-vanishing gradient-like fields on $X$. \smallskip

Evidently $gc^+(X) \geq c^+(X)$. Let $M^\circ$ denote the M\"{o}bius band.  In Section 6, we will show that $gc^+(M^\circ) = 1$, while $c^+(M^\circ) = 0$, so the two notions of complexity are different. 
 \smallskip

In terms of this complexity, we can restate the Lemma 3.1 as follows.

\begin{corollary}\label{cor 4.1} Let $X$ be a connected compact surface with boundary. Let $v$ be a boundary generic vector field on $X$, subject to the condition $\mathsf{Ind}(v) = 0$.

Then the complexity of $v$ satisfies the inequality  $$c^+(v) \geq 2\cdot \dim_\R H_1(X; \R) - 2 = -2\cdot \chi(X).$$ When $\chi(X) \leq 0$, this inequality turns into the equality $c^+(v) = -2\cdot \chi(X)$ if and only if $v$ is boundary concave.  

As a result, for any natural $N$, there are finitely many connected compact surfaces  of bounded complexity $c^+(X) \leq N$. In fact, the number of such surfaces (counted up to a homeomorphism) grows as a quadratic function in $N$. \hfill $\diamondsuit$
\end{corollary}
\noindent {\bf Example 4.1.} For any non-vanishing boundary concave field $v$ on the torus with a single hole, $\#\{\d_2^+X(v)\} = 2$. In fact, the constant field $v$, being restricted to the complement to a convex disk in $T^2$, is boundary concave and has the property $\#\{\d_2^+X(v)\} = 2$. Thus, by Corollary 4.1,  $c^+(X) = 2$.  \hfill $\diamondsuit$
 \smallskip

Lemma \ref{lem3.1}  leads immediately to

\begin{corollary}\label{cor4.2} Let $X$ be a sphere with $g$ handles and $k$ holes, where $g, k \geq 1$. If $X$ admits a non-vanishing boundary concave field $v$, then $\#\{\d_2^+X(v)\} =  4g - 4 + 2k$. \hfill $\diamondsuit$
\end{corollary}

Given a compact surface $X$ with boundary, we form its \emph{double} $DX =_{\mathsf{def}} X \cup_{\d X} X$ by attaching two copies of $X$ along their boundaries. Note that $\chi(DX) = 2\cdot \chi(X)$. Therefore,  $\chi(X) < 0$ if and only if $\chi(DX) < 0$. \smallskip

Recall that any closed orientable surface with a negative Euler number admits a metric of constant negative curvature $-1$. So if $\chi(X) < 0$, then $DX$ admits such hyperbolic metric. \smallskip

Let $vol(DX)$ denote the hyperbolic volume of $DX$, and let $vol(\D^2)$ denote the volume of an ideal hyperbolic triangle $\D^2$ in the hyperbolic plane $\mathbf H^2$. 

In $2D$, a remarkable convergence of topology and  geometry takes place. In the spirit of this convergence, since $\chi(DX) = -vol(DX)/ vol(\D^2)$, Corollary 4.1 admits a more geometric reformulation:

\begin{theorem}\label{th1.3} Let $v \neq 0$ be a boundary generic vector field on a compact  connected and orientable surface $X$ with boundary. Assume that $\chi(X) < 0$\footnote{This excludes disk and annulus.}.
Then the complexity of the $v$-flow satisfies the inequality: $$c^+(v)\; \geq \;  vol(DX)/ vol(\D^2).$$ Moreover, $c^+(v) = vol(DX)/ vol(\D^2)$ if and only if $v$ is boundary concave. 
\hfill $\diamondsuit$
\end{theorem}
Theorem \ref{th1.3} admits far reaching multidimensional generalizations (see  \cite{AK}, \cite{K5}). They are valid for so called \emph{traversally generic} vector fields  (see Definitions 5.1 and 5.2 and \cite{K2}) on arbitrary smooth compact $(n+1)$-dimensional manifolds $X$ with boundary. Such fields $v$ naturally generate stratifications of trajectory spaces $\mathcal T(v)$, whose strata are labeled by the combinatorial patterns of tangency from the universal partially ordered set $\Omega^\bullet_{'\langle n]}$ (see the end of Section 6 and \cite{K3}).  In high dimensions, we use the simplicial semi-norms $\|\sim\|_\D$ of Gromov \cite{Gr} on the homology $H_\ast(X; \R)$ and $H_\ast(DX; \R)$ (as a substitute of the hyperbolic volume) to provide lower bounds on the number of connected components of  the $\Omega^\bullet_{'\langle n]}$-strata of any given dimension.  
\smallskip


\section{On spaces of vector fields}
\begin{definition} We say that a vector field $v \neq 0$ on a compact surface $X$ is \emph{traversing} if all its trajectories are closed segments or singletons\footnote{It easy to see that the ends of these segments, as well as the singletons, reside in $\d X$.}. \hfill $\diamondsuit$
\end{definition}

Each trajectory $\g$ of a traversing field $v$ must reach the boundary both in positive and negative times: otherwise $\g$ is not homeomorphic to a closed interval. \smallskip

We denote by $\mathcal V_{\mathsf{trav}}(X)$ the space  (in the $C^\infty$-topology) of all traversing fields on $X$. \smallskip

We denote by $\mathcal V_{\mathsf{grad}}(X)$ the space (in the $C^\infty$-topology) of all gradient-like  fields on a given compact surface $X$  
and  by $\mathcal V_{\neq 0}(X)$ the space of all non-vanishing fields on $X$. \smallskip

The next lemma says  that $v$ is traversal if and only if it is non-vanishing and of a gradient type (see \cite{K1} for the proof).

\begin{lemma}\label{5.1} For any compact connected surface $X$ with boundary, $$\mathcal V_{\mathsf{trav}}(X) = \mathcal V_{\mathsf{grad}}(X) \cap \mathcal V_{\neq 0}(X).$$ 
\hfill $\diamondsuit$
\end{lemma}

The surfaces $X$ and vector fields $v$ we consider are all smooth. We can add an external collar to $X$ to form a diffeomorphic surface $\hat X \supset X$ and to extend $v$ to a smooth field $\hat v$ on $\hat X$. Let $\hat \g$ be a $\hat v$-trajectory (or rather its germ) through a point $x$ of $\d X$. We can talk about \emph{order of tangency} of two smooth curves,  $\hat\g$ and $\d X$, at $x \in \hat\g \cap\d X$ in $\hat X$ (see Definition \ref{def1.5}). We say that the tangency of $\hat\g$ to $\d X$ is \emph{simple} if  its degree is 2. When the two curves are transversal at $x$ we say that the order of tangency is $1$. In fact, this notions depend only on $(X, v)$ and not on the extension $(\hat X, \hat v)$. 

\begin{definition}\label{def1.4} A traversing vector field $v$ on a compact surface $X$ is called \emph{traversally generic}, if two properties are valid: {\bf(1)} if a  trajectory $\g$ is tangent to the boundary $\d X$, then the tangency is simple, and {\bf (2)} no $v$-trajectory $\g$ contains more then one simple point of tangency to $\d X$.\footnote{In particular, a traversally generic $v$ is boundary generic.} \hfill $\diamondsuit$
\end{definition}

We denote by $\mathcal V^\ddagger(X)$ the space of all traversally generic vector fields on a compact surface $X$. In fact, the notion of traversally generic field is available in any dimension (see \cite{K2}).\smallskip

As the name suggests, the traversally generic fields are typical among all traversing fields; furthermore, a perturbation of any traversally generic field is traversally generic. This is the content of the next theorem. Its validation requires an involved argument, which even in $2D$ resists a significant simplification \cite{K2}. 
\begin{theorem}\label{th1.4}  For any compact connected surface $X$ with boundary, the space $\mathcal V^\ddagger(X)$ traversally generic fields is open and dense in the space $\mathcal V_{\mathsf{trav}}(X) = \mathcal V_{\mathsf{grad}}(X) \cap \mathcal V_{\neq 0}(X)$. \hfill $\diamondsuit$
\end{theorem}

\section{Graph-theoretical approach to the concavity  of  traversing fields in 2D}

We start with a couple of very natural questions. \smallskip

\noindent {\bf Question 6.1.}
Which compact connected surfaces with boundary admit  boundary concave gradient-like vector fields $v \neq 0$? \hfill $\diamondsuit$
\smallskip

Recall that $c^+(X) \leq gc^+(X)$.
\smallskip

\noindent {\bf Question 6.2.}
Are there compact connected surfaces $X$ with boundary  for which $c^+(X) < gc^+(X)$? \hfill $\diamondsuit$
\smallskip


On many occasions we took advantage of the fact that, for traversally generic vector  fields $v$, the trajectory spaces $\mathcal T(v)$ are finite graph whose verticies have valency $1$ and $3$ only (see Fig. 4). Moreover, for a traversally generic boundary concave field $v$, all the verticies of $\mathcal T(v)$ have valency $3$.  Now we will take a closer look at the graph-theoretical models of the boundary concave and traversally generic fields in 2D.

Let $G$ be a finite connected trivalent graph with $a$ verticies. We denote by $\b G$ its barycentric subdivision: each edge $e$ of $G$ is divided by a new vertex $v_e$,  its center. We consider the finite set $\mathsf{Tri}(G)$ of all \emph{colorings} of the edges of $\b G$ with \emph{tree} colors so that, at each vertex of $G$, exactly three distinct colors are applied. Thus, $\#\mathsf{Tri}(G) = 6^a$.

\begin{theorem}\label{th1.5} Let $G$ be a finite connected trivalent graph. Each coloring $\a \in \mathsf{Tri}(G)$ produces (in a canonical way) a compact connected surface $X(G, \a)$ with boundary. The surface $X(G, \a)$ admits a traversally generic \emph{concave} vector field $v(G, \a)$. The cardinality of the locus $\d_2^+X(G, \a)\big(v(G, \a)\big)$ is the number of verticies in $G$. 

Moreover, every connected surface with boundary,  which admits a traversally generic concave  vector field, can be produced in this way. 
\end{theorem}

\begin{figure}[ht]\label{fig1.6}
\centerline{\includegraphics[height=3.5in,width=4in]{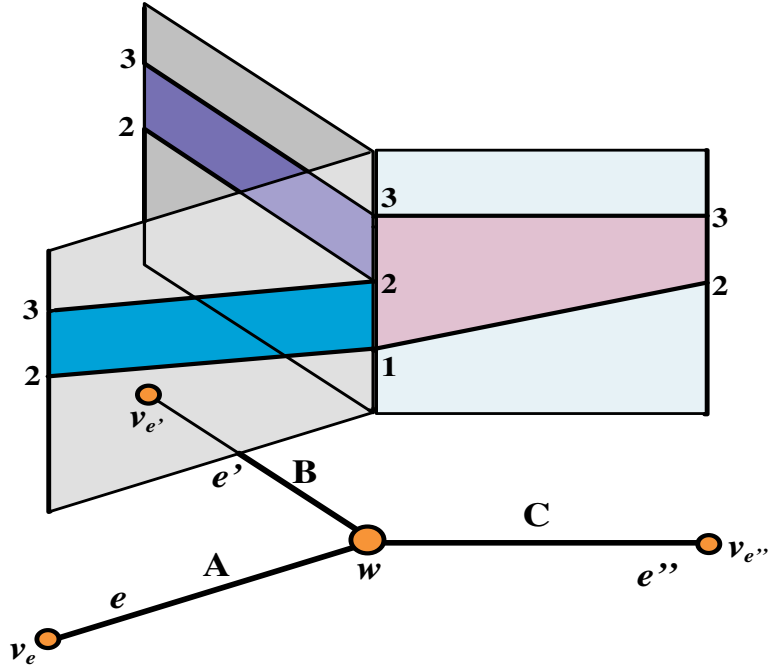}}
\bigskip
\caption{\small{Constructing the surface $X(G, \a)$ in the vicinity of a trivalent vertex $w \in G$. The verticies $v_e, v_{e'}, v_{e''} \in \b G$ are the centers of edges $e, e', e''$ of $G$.}}
\end{figure}

\begin{proof} Let $\mathsf A, \mathsf B, \mathsf C$ denote the three distinct colors, and $\mathcal P = \{\mathsf A, \mathsf B, \mathsf C\}$ the entire pallet. 

Consider a $2$-dimensional space $Z = G \times (0, 4)$. It has singularities in the form of binders of the-page open books (see Fig.1.6 
). The binders correspond to the verticies of $G$.

First, employing a given coloring $\a$, we will construct a piecewise linear surface $\hat X(G, \a) \subset Z$. The vector field on $\hat X(G, \a)$ will be induced by the product structure in $Z$. 

For each edge $e \subset G$ and its barycenter $v_e \in \b G$, we place the interval $v_e \times [2, 3]\subset Z$ over $v_e$. Let $\hat e$ be half of the interval $e \subset G$, bounded two vericies  $v_e \in \b G$ and $w \in G$. Over $\hat e$, we place a strip $E \subset Z$; its construction depends on the color attached to the interval $[v_e, w]$ as follows:
\begin{itemize}
\item if the color of $[v_e w]$ is $\mathsf A$, then we link the vertex $v_e \times 2$ with the vertex $w \times 1$ by a line in the rectangle $R = [v_e, w] \times (0, 4)$, and  the vertex $v_e \times 3$ with the vertex $w \times 2$ by another line in $R$;
\item if the color of $[v_e, w]$ is $\mathsf B$, then we link by a line in $R$ the vertex $ v_e \times 2$ with the vertex $w \times 2$, and  the vertex $v_e \times 3$ with the vertex $w \times 3$ by another line;  
\item if the color of $[v_e, w]$ is $\mathsf C$, then we link by a line in $R$ the vertex $v_e \times 2$ with the vertex $w \times 1$, and  the vertex $v_e \times 3$ with the vertex $3 \times w$ by another line.
\end{itemize}
By definition, $E(e, w)$ is the strip in $[v_e, w] \times (0, 4)$, bounded by the two lines whose construction is has been described above. Thanks to the monotonicity of the bijections $A: \{2, 3\} \to \{1, 2, 3\}$,  $B: \{2, 3\} \to \{1, 2, 3\}$, and $C: \{2, 3\} \to \{1, 2, 3\}$ that correspond to the colors $\mathsf A, \mathsf B, \mathsf C$, the lines that bound the strips $E(e, w)$ do not intersect.  We denote by $\hat X(G, \a)$ the union  of all such strips. 

The local model of each binder implies that indeed $\hat X(G, \a)$ is piecewise linear surface, imbedded in the singular space $Z$. Inside $Z$, one can smoothen the sharp edges of the boundary $\d \hat X(G, \a)$ in order to get a smooth surface $X(G, \a)$ (can you visualize this smoothing in the vicinity of point $w \times 2$ from Fig. 6?). The restriction of the product structure in $Z$ to its subspace $X(G, \a)$ produces a smooth non-vanishing vector field $v(G, \a)$ on $X(G, \a)$. Its trajectories (the vertical lines in $Z$) will be simply tangent to $\d X(G, \a)$ exactly at the points of the type $2 \times w$, where $w$ runs over the set of verticies of $G$. By Theorem \ref{th1.3}, this field $v(G, \a)$ is of the gradient type. \smallskip

Conversely, any traversally generic and concave  vector field $v$ on a connected compact surface $X$ with boundary, produces a map $\Gamma: X \to \mathcal T(v)$, where the space of trajectories is a finite trivalent graph. Its verticies are in 1-to-1 correspondence with  the points of the locus $\d_2^+X(v)$. 

As a point $v_e$ in the open edge $e$ of the graph $\mathcal T(v)$ approaches a vertex $w$, the intersection of the $v$-trajectory $\g = \Gamma^{-1}(v_e)$ with the boundary $\d X$ defines a bijection of the $v$-ordered set $\g \cap \d X$ of cardinality $2$ to a $v$-ordered set $\Gamma^{-1}(w) \cap \d X$  of cardinality $3$, the orders being respected by the bijections. This determines one of three colors we attach to the half-edge $[v_e, w]$. Therefore the geometry of the flow determines a tricoloring of the graph $\b \mathcal T(v)$.   
\end{proof}

The next theorem answers Questions 6.1 and 6.2.

\begin{theorem}\label{th6.2}
\begin{itemize}
\item For any orientable connected surface $X$ with boundary,  the two complexities are equal: $gc^+(X) = c^+(X)$. Moreover, any such $X$, but the disk, admits a boundary concave traversally generic vector field.  As a result,  for any orientable connected $X$ with boundary, but the disk,  the two complexities are equal to  $-2\chi(X)$. 

\item For any non-orientable connected surface $X$ with boundary, which is a boundary connected sum of several punctured Klein bottles and annuli, $gc^+(X) = c^+(X) = -2\chi(X)$ as well. Again, any such $X$  admits a boundary concave traversally generic vector field. 

\item In contrast, the M\"{o}bius band $M^\circ$ does not admit a  boundary concave traversally generic vector field.  In fact,  $c^+(M^\circ) = 0$ and $gc^+(M^\circ) = 1$. 
\item Moreover, $c^+(M^\circ \#_\d X) \leq c^+(X) + 2$ and $gc^+(M^\circ \#_\d X) \leq gc^+(X) + 3$ for any $X$ as in the first two bullets\footnote{These inequalities, together with the computations of the complexities in the first three bullets, cover the entire variety of compact connected surfaces with boundary.}. 
\end{itemize}
\end{theorem}

\begin{proof} Consider the boundary connected sum $X_1 \#_\d X_2$ of two compact surfaces $X_1$ and $X_2$.  The Euler number of the sum satisfies the rule $$\chi(X_1 \#_\d X_2) = \chi(X_1) + \chi(X_2) - 1.$$  On the other hand, given two boundary generic fields $v_1$ and $v_2$, there exists a traversally generic field $w$ on $X_1 \#_\d X_2$  such that 
$$|\d_2^+(X_1 \#_\d X_2)(w)| = |\d_2^+(X_1)(v_1)|  + |\d_2^+(X_2)(v_2)|  + 2.$$
Indeed,  we may attach a $1$-handle $H$ to $\d_1^-X_1(v_1) \coprod \d_1^+X_2(v_2)$ so that an $H$ has a neck with respect to the extension $w$.  Such field $w$ contributes two points to $\d_2^+(X_1 \#_\d X_2)(w)$. Of course, this construction fails when $\d_1^-X_1(v_1) \coprod \d_1^+X_2(v_2) = \emptyset$; however, for traversing fields $v$, both loci 
$\d_1^\pm X(v) \neq \emptyset$. 

By Corollary 4.1, 
 if $X$ admits a boundary concave field, then $c^+(X) = -2 \chi(X)$, provided $\chi(X) \leq 0$. In particular, if $X$ with a non-positive Euler number admits a \emph{boundary concave} traversally generic $v$, then $$gc^+(X) = c^+(X) = -2 \chi(X).$$ 

Let $v_1$ and $v_2$ be some boundary generic/ traversally generic  fields which deliver the two gradient complexities.  The previous arguments about extending  $v_1$ and $v_2$ across the handle $H$ imply that if $gc^+(X_1) = -2\chi(X_1)$ and $gc^+(X_2) = -2\chi(X_2)$ (say both surfaces admit boundary concave and traversally generic fields), then $$gc^+(X_1 \#_\d X_2) \leq gc^+(X_1) + gc^+(X_2) +2 = -2\chi(X_1 \#_\d X_2),$$
provided that $\chi(X_1 \#_\d X_2) \leq 0$.
Since the reverse inequality holds by Corollary 4.1, 
we get  $$gc^+(X_1 \#_\d X_2) = -2 \chi(X_1 \#_\d X_2)$$ when $gc^+(X_1) = -2\chi(X_1)$ and $gc^+(X_2) = -2\chi(X_2)$. \smallskip

Recall the topological classification of closed connected surfaces. Any such surface is either a sphere, or a connected sum of several tori (the orientable case), or  a  connected sum of several projective spaces (the non-orientable case). Therefore any connected surface with boundary is obtained from the surfaces in this list by deleting at least one disk. 

Let $T^\circ$ denote the complement to an open disk in a $2$-torus, and $M^\circ$ denote the complement to an open disk in a projective plane---the M\"{o}bius band---, and let $A$  denote the annulus.  Thus any connected surface with boundary is either  a disk $D$, or a boundary connected sum of several copies of punctured tori $T^\circ$ and annuli $A$ (the orientable case), or a boundary connected sum of several copies of  M\"{o}bius bands $M^\circ$ and annuli $A$ (the non-orientable case). 

Let us now compute the complexities of the basic blocks in this decomposition. Note that $c^+(D) = 0 = gc^+(D)$, since $D$ admits a convex traversing flow. Also $c^+(A) = 0 = gc^+(A)$, the latter equality being delivered by the radial gradient field. 

We claim that $c^+(T^\circ) = 2 = gc^+(T^\circ)$. Indeed, since $\chi(T^\circ) = -1$, by  Corollary 4.1, 
we get $c^+(T^\circ) \geq 2$. On the other hand, there exists a trivalent graph $G_T$ with an appropriate tricoloring  and exactly \emph{two} verticies such that, applying the construction from Theorem \ref{th1.5}, we produce a traversally generic field $v(G_T, \a)$ on the surface $X(G_T, \a) = T^\circ$ with the cardinality $2$ locus $\d_2^+X(G_T, \a)(v(G_T, \a))$. As a result, both complexities of $T^\circ$ equal to $2$. 

Similar considerations apply to the punctured Klein bottle $K^\circ = M^\circ \#_\d  M^\circ$ and a different trivalent graph $G_K$ with two verticies and an appropriate tricoloring. Since $\chi(K^\circ) = -1$, we conclude that  $c^+(K^\circ) = 2 = gc^+(K^\circ)$.

The third trivalent graph $G_A$ with two verticies and an appropriate tricoloring delivers a traversally generic boundary concave flow on a punctured annulus $A^\circ$, the disk with two holes. Thus, $c^+(A^\circ) = 2 = gc^+(A^\circ)$. 

In fact, Theorem \ref{th1.5} implies that $T^\circ, K^\circ, A^\circ$ are the only connected surfaces of the gradient complexity $2$ that admit concave traversally generic fields. Indeed, just start with the tree ``$>\bullet-\bullet<$" with two trivalent verticies and consider the ways one can identify its four leaves in pairs. Then consider all admissible tricologings of the resulting graphs $G$. This cases will deliver the three model tricolored graphs $G_T, G_K, G_A$.  
\smallskip

Now the ``quasi-additivity" of Euler numbers and gradient complexities under the connected sum operations imply that the gradient complexity of boundary connected sums $$X = (T^\circ \#_\d \dots \#_\d T^\circ) \#_\d (A \#_\d \dots \#_\d A) \#_\d(K^\circ \#_\d \dots \#_\d K^\circ)$$ of several copies of the model surfaces $T^\circ, K^\circ, A$ is equal to $2|\chi(X)|$.  Indeed, these properties imply that  $gc^+(X) \leq 2 \cdot |\chi(X)|$, while in general  $gc^+(X) \geq 2 \cdot |\chi(X)|$. Moreover, every such surface $X$ admits a \emph{boundary concave} traversally generic field (by the $1$-handle-with-a-neck argument), since the basic blocks $T^\circ, K^\circ$ and $A$ do. 
\smallskip

The M\"{o}bius band $M^\circ$ is different. We notice that $M^\circ$ admits a non-vanishing vector field $v$ with a single closed trajectory---the core of the M\"{o}bius band---and  transversal to the boundary $\d M^\circ$. Thus, $c^+(M^\circ) = 0$. Now consider a trivalent graph $G_M$ with a single vertex of valency $3$ and a single vertex of valency $1$ (this $G_M$ is a circle to which a radius is attached). The construction from Theorem \ref{th1.5} applies to produce a remarkable embedding of the  M\"{o}bius band in the product $G_M \times [0, 4]$. So we conclude that $M^\circ$ admits a traversally generic field $v$ (not concave!) with $\d_2^+X(v)$ being a singleton ($\d_2^-X(v)$ is a singleton as well).  As a result, $gc^+(M^\circ) \leq 1$. On the other hand, any traversally generic field $v$ on $M^\circ$ must produce the graph $\mathcal T(v)$ which is homotopy equivalent to a circle, the homotopy type of $M^\circ$. If $gc^+(v) = 0$,  this graph $\mathcal T(v)$ has no trivalent verticies, in which case,  $\mathcal T(v)$ is homeomorphic to a circle. So $M^\circ \to \mathcal T(v)$ must be a fibration whose fibers (the $v$-trajectories) are segments. Moreover, thanks to the field $v$, this fibration is orientable, a  contradiction with the non-orientability of $M^\circ$. Therefore, we conclude  that $gc^+(M^\circ) = 1$,  while $c^+(M^\circ) = 0$. \smallskip

Finally,  for any $X$ which is a boundary connected sum of $T^\circ$'s, $K^\circ$'s, and $A$'s, by the same arguments,  the inequalities  $c^+(M^\circ \#_\d X) \leq -2\chi(X) + 2$ and $gc^+(M^\circ \#_\d X) \leq -2\chi(X) + 3$ hold. This validates the claim in the last bullet.
\end{proof}%
%
\section{Combinatorics of tangency for traversing flows in 2D}

Pick an extension $\hat X$ of a given compact surface $X$ by adding an external collar to $X$.  Let $\hat v$ be an extension of a given field $v$ into $\hat X$.  Pick a smooth auxiliary function $z: \hat X \to \R$ such that: 

\begin{itemize}
\item $0$ is a regular  value  of $z$,
\item $z^{-1}(0) = \d X$,
\item $z^{-1}((-\infty, 0]) = X$,
\end{itemize}
\begin{eqnarray}\label{eq1.6}
\end{eqnarray}

\begin{definition}\label{def1.5} Let $\hat\g$ be a $\hat v$-trajectory  through a point $x \in \d X$. We say that  $\hat\g$ has the \emph{order/multiplicity of tangency} $k$ to $\d X$ at $x$, if  $\mathcal L_{\hat v}^{\{j\}} (z) = 0$ for all $j < k$, and $\mathcal L_{\hat v}^{\{k\}} (z) \neq 0$ at $x$ \footnote{this is equivalent to saying that the $(k-1)$-st jet at $x$ of $z|_\g$ vanishes, but the $k$-th jet does not.}. Here $\mathcal L_{\hat v}^{\{j\}} (z)$ denotes the $j^{th}$ iterated $\hat v$-directional derivative of the function $z$. \hfill $\diamondsuit$
\end{definition}

Given a traversally generic vector field $v$ on a compact connected surface $X$, 
we will attach the \emph{combinatorial pattern} $(1,1)$ to a typical $v$-trajectory $\g \subset X$ that corresponds to the edges of the graph $\mathcal T(v)$, the pattern $(121)$ to the trajectories that correspond to the trivalent verticies of $\mathcal T(v)$, and the pattern $(2)$ to the univalent verticies (see Fig. 4).  In fact, the numbers 1 and 2 in these patterns reflect the order of tangency of the curves $\hat\g$ and $\d X$ at the points of $\g \cap \d X$ (see Definition \ref{def1.5}). On a given compact surface $X$, for \emph{traversally generic}  fields $v$ no other patterns (say, like $(1221)$ or $(13)$) occur. In $2D$, this conclusion follows from Definition \ref{def1.5}. 
\smallskip

The lemma below is another way to state this fact. Its proof, relying on the Malgrange Preparation Theorem \cite{Mal}, can be found in \cite{K2}.

\begin{lemma} Let $v$ be a traversally generic field on $X$. Extend $(X, v)$ to a pair $(\hat X, \hat v)$. In the vicinity of each $v$-trajectory $\g$, there exist special local coordinates $(u, x)$ in $\hat X$ and a real polynomial $P(u, x)$ of  degree $2$ or $4$ such that:
\begin{itemize}
\item each $\hat v$-trajectory is given by the equation $\{x = const\}$,   
\item the boundary $\d X$ is given by the polynomial equation $\{P(u, x) = 0\}$,
\item $X$ is given by the polynomial inequality  $\{P(u, x) \leq 0\}$.
\end{itemize}

The polynomial $P(u, x)$ takes  three canonical forms:
\begin{enumerate}
\item $u(u -1)$, which corresponds to the combinatorial pattern $\mathbf{(11)}$,
\item $u^2 - x,$ which corresponds to the combinatorial pattern $\mathbf{(2)}$,
\item $u\big((u -1)^2 + x\big)(u -2),$ which corresponds to the  pattern $\mathbf{(121)}$. \hfill $\diamondsuit$
\end{enumerate}
\end{lemma}

To summarize, at $\d_2X(v)$ the order of tangency is $2$; the trajectories through $\d_2^+X(v)$ have the combinatorial tangency pattern $(121)$, and through $\d_2^-X(v)$ the combinatorial tangency pattern $(2)$. The rest of trajectories have the pattern $(11)$. \smallskip 

We denote by $\Omega^\bullet_{'\langle 1]}$ the partially ordered set whose elements are $(11), (2), (121)$ and the order is defined by $(11) \succ (2)$ and $(11) \succ (121)$. This combinatorics does not look impressive. However, in higher dimensions, traversally generic fields on $(n+1)$-manifolds with boundary generate a  rich and interesting partially ordered finite list $\Omega^\bullet_{'\langle n]}$ of combinatorial tangency patters. The poset $\Omega^\bullet_{'\langle n]}$ is \emph{universal} in each dimension $n+1$. They are discussed in \cite{K3}.

\section{Holography of traversing flows on surfaces}
Let $v$ be a traversing and boundary generic vector field on a compact connected surface $X$ with boundary.  For any point $z \in \d_1^+X(v)$, consider the closest point $w(z) \in \d_1^-X(v)$ that can be reached by moving along the trajectory $\g_z$ through $z$ in the direction of $v$ (see Fig. 7). Note that $w(z) = z$ if and only if $z \in \d_2^-X(v)$. 

The correspondence $z \to w(z)$ defines a map $$C_v: \d_1^+X(v) \to \d_1^-X(v)$$
which we call the \emph{causality map}. It is a distant relative of the classical Poincar\'{e} Return Map. 

Alternatively, one can think of $C_v$ as determining a \emph{partial order} ``$z \prec w(z)$" among the points of the boundary $\d X$.\smallskip

The word ``causality" in the name of $C_v$ is motivated by the following pivotal special case. \smallskip

\begin{figure}[ht]\label{fig1.7}
\centerline{\includegraphics[height=2.7in,width=4in]{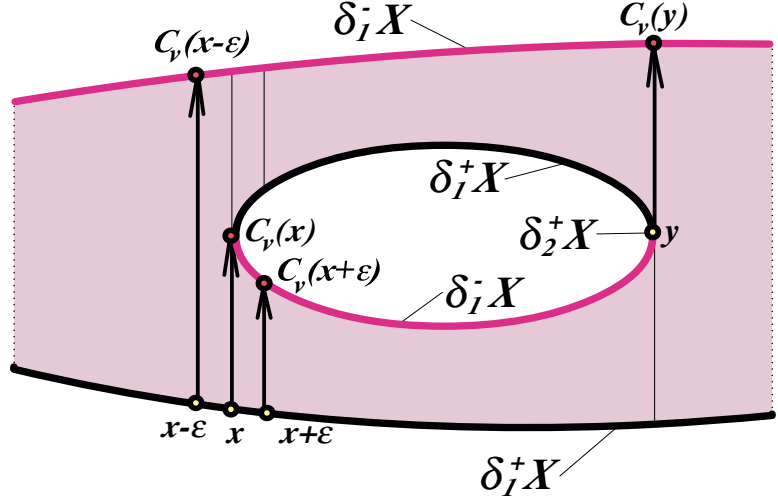}}
\bigskip
\caption{\small{An example of the causality map $C_v: \d_1^+X(v) \to \d_1^-X(v)$.} Note the discontinuity of $C_v$ in the vicinity of $x$. }
\end{figure}

\noindent {\bf Example 8.1. }
Let $w = w(\theta, t)$ be a smooth time-dependent  vector field on the circle $S^1$ (equipped with the angular coordinate $\theta$). It gives rise to a vector field $v = (w, 1)$ on the cylinder $S^1 \times \R$. We think about the factor $S^1$ as  \emph{space} and about the factor $\R$ as \emph{ time} $t$. So we call $S^1 \times \R$ \emph{the space of events}. Note that $v \neq 0$ is a gradient-like field with respect to the time function $T: S^1 \times \R \to \R$.

Pick any smooth compact and connected surface $X \subset S^1 \times \R$. Such a surface has a boundary $\d X$. We call $X$ the  \emph{event domain}, and its boundary $\d X$ the \emph{event horizon}. 

Since the field $v \neq 0$ is traversing in $X$, the map $C_v$ is well-defined.  Then the map $C_v: \d_1^+X(v) \to \d_1^-X(v)$ indeed gives rise the \emph{causality relation}  on the event horizon: the correspondence $C_v$ reflects the \emph{evolution} of an event $z$ into the event $C_v(z)$. 
\hfill $\diamondsuit$
\smallskip

Let $\mathcal C(\d_2^+X(v))$ denotes the union of $v$-trajectories  through the points of the concavity locus $\d_2^+X(v)$. 

The causality map is \emph{discontinuous} at the points of the intersection $\mathcal C(\d_2^+X(v)) \cap \d_1^+X(v)$ (see Fig. 7). On the positive side, the discontinuities of the causality map $C_v$ are not too bad: in a sense, the map has ``left" and ``right" limits. 
\smallskip 

Given a pair $(X, v)$, the $v$-trajectories, viewed as \emph{un}parametrized $v$-oriented curves, produce an oriented $1$-dimensional foliation $\mathcal F(v)$ on $X$.  

\begin{theorem}\label{th7.1}{\bf (The Causal Holography Principle in 2D).} 

Let $(X_1, v_1)$ and $(X_2, v_2)$ be two compact connected surfaces with boundaries,  carrying  traversally generic vector fields $v_1$ and $v_2$, respectively. Assume that there is a diffeomorphism $\Phi^\d : \d X_1 \to \d X_2$ which conjugates the two causality maps:
$$C_{v_2} \circ \Phi^\d = \Phi^\d \circ C_{v_1}.$$

Then $\Phi^\d$ extends to a diffeomorphism $\Phi: X_1 \to X_2$ which maps the oriented foliation $\mathcal F(v_1)$ to the oriented foliation $\mathcal F(v_2)$.
\end{theorem}

\begin{proof} We will only sketch the argument. A fully developed proof of the multidimensional analogue of  this theorem is contained in \cite{K4}. 

First, we notice that since $C_{v_1}$ and $C_{v_2}$ are $\Phi^\d$-conjugate, the diffeomorphism $\Phi^\d$ induces a well-defined continuous map $\Phi_{\mathcal T}: \mathcal T(v_1) \to \mathcal T(v_2)$ of the trajectory spaces. Moreover, $\Phi_{\mathcal T}$ preserves the stratifications of the  two trajectory spaces/graphs  by the combinatorial type of trajectories. That is, the trivalent verticies of  $\mathcal T(v_1)$ are mapped to the trivalent verticies of  $\mathcal T(v_2)$, the univalent verticies are mapped to univalent verticies, and the interior of the edges to the interior of the edges.

Then we pick a smooth function $f_2: X_2 \to \R$ such that $df_2(v_2) > 0$. With the help of $\Phi^\d$, we pull-back $f_2|_{\d X_2}$ to get a smooth function $f_1^\d: \d X_1 \to \R$ such that  $$f_1^\d(z) < f_1^\d(C_{v_1}(z))$$ for all $z \in \d_1^+X_1(v_1)$. 

Then we argue that $f_1^\d$ extends to a smooth function $f_1$ such that $df_1(v_1) > 0$. 

We use $f_1$ to embed $X_1$ in the product  $\mathcal T(v_1) \times \R$ by the formula $$\a_{(v_1, f_1)}(z) =  (\g_z, f_1(z)),$$ where $\g_z$, the $v_1$-trajectory through $z$, is viewed as the point $\Gamma_1(z)$ of the graph $\mathcal T(v_1)$. Similarly, we use $f_2$ to embed the surface $X_2$ in the product  $\mathcal T(v_2) \times \R$ with the help of  the map $\a_{(v_2, f_2)}$.

Finally, we employ $\Phi_{\mathcal T}$, $f_1$ and $f_2$ to construct a map $$\hat \Phi: \mathcal T(v_1) \times \R \to  \mathcal T(v_2) \times \R$$ by the formula $$\hat\Phi(\g, t) = \big(\Phi_{\mathcal T}(\g),\, f_1(f_2^{-1}(t))\big),$$ where $t$ belongs to the $f_2$-image of the trajectory $\Gamma_2^{-1}\big(\Phi_{\mathcal T}(\g)\big)$. 

Crudely, the restriction of $\hat\Phi$ to $\a_{v_1, f_1}(X_1) \subset \mathcal T(v_1) \times \R$ is the desired diffeomorphism $\Phi: X_1 \to X_2$. 

Note that, in general, the pull-back $\Phi^\ast(f_2)$ is not $f_1$; so the parametrizations of  the trajectories are not respected by the diffeomorphism $\Phi$, but the $1$-foliations $\mathcal F(v_1)$ and $\mathcal F(v_2)$ are.
\end{proof}

\begin{corollary}\label{cor7.1} Let $X$ be a compact connected surface with boundary, and $v$ a smooth traversally generic vector field on it. 

Then the knowledge of the causality map $C_v: \d_1^+X(v) \to \d_1^-X(v)$ is sufficient for a reconstruction of the pair $(X, \mathcal F(v))$, up to a diffeomorphism that is constant on $\d X$.  \hfill $\diamondsuit$
\end{corollary}

The world ``holography" is present in the name of Theorem \ref{th7.1} since the surface $X$ and the $2D$-dynamics of the $v$-flow in it are recorded on two $1$-dimensional screens, $\d_1^+X(v)$ and $\d_1^-X(v)$. \smallskip

Theorem \ref{th7.1} and Corollary \ref{cor7.1} are valid in any dimension (\cite{K4}).\smallskip

\noindent {\bf Example 8.2.}
Let $v$ be a traversally generic field on a connected surface $X$ whose boundary $\d X$ is a \emph{single} loop. Then the boundary $\d X$ is divided into  $q$  disjoint arcs $a_1, \dots , a_q$  that form $\d_1^+X(v)$ and $q$ complementary arcs $b_1, \dots , b_q$ that form $\d_1^-X(v)$. The causality map $$C_v: \coprod_{i=1}^q a_i \to \coprod_{i=1}^q b_i$$ can be represented by its graph $G(C_v) \subset \prod_{i, j} a_i \times b_j$. 

The map $C_v$ (the curve $G(C_v)$) is discontinuous at exactly $c^+(v)$ points in $\coprod_{i=1}^q a_i $ that correspond to the points of the intersection $\mathcal C(\d_2^+X(v)) \cap \d_1^+X(v)$. There the map $C_v$ has distinct  left and  right limits. 

According to the Corollary \ref{cor7.1}, the curve $G(C_v) \subset \prod_{i, j} a_i \times b_j$ determines $X$ and the \emph{un}-parametrized dynamic of the $v$-flow, up to a diffeomorphism $\Phi: X \to X$ that is the identity on $\d X$. Note that the number $q$ alone is not sufficient even to determine the genus of the surface $X$.
\hfill $\diamondsuit$
\smallskip

Revisiting Example 8.2, we get the following interpretation of Corollary \ref{cor7.1}:

\begin{corollary}\label{cor7.2} For any smooth time-dependent vector field $w$ on the circle $S^1$, the causality relation on the event horizon $\d X$ is sufficient for a reconstruction of the event domain $X$ and the un-parametrized dynamics of the $(w, 1)$-flow, up to a diffeomorphism of $X$ that is the  identity on $\d X$. \hfill $\diamondsuit$
\end{corollary}

The theory of billiards on Riemmanian surfaces $X$ with boundary benefits from applying the $3D$-version of the Causal Holography to the geodesic flow on the $3$-fold $SX$, the space of unit tangent vectors on $X$. See \cite{K4} for some of these applications. In addition to geodesic billiards, they include the classic inverse geodesic scattering problems. 

\section{Convex quasi-envelops and characteristic classes of traversing flows on orientable surfaces}

Traversing flows have interesting \emph{characteristic classes}---elements of certain cohomology---associated with them. In dimension two, they are quite primitive, but for high-dimensional flows, surprisingly rich (see \cite{K6}). \smallskip

We have seen that the traversally generic flows exhibit a very particular combinatorial patterns of tangency to the boundary $\d X$. In particular, for generic $2D$-flows, no tangencies of orders $\geq 3$ occur. \smallskip

There is a nice link between this behavior and the \emph{spaces of smooth functions} $f: \R \to \R$ or even polynomials that have no zeros of multiplicities $\geq 3$. 
To explain the connection, we will need the following definition/construction. \smallskip

Let  $\mathcal F$ denote the space (in the $C^\infty$-topology) of smooths functions $f: \R \to \R$ which are identically $1$ outside of a compact set. Let $\mathcal F_{\leq 2}$ be its subspace, formed by functions  that have zeros only of multiplicity $\leq 2$. 

Such spaces of functions with ``\emph{moderate singularities}" have been studied in depth by V. I. Arnold \cite{Ar} and V. A. Vassiliev \cite{V}. In $2D$,  we  employ just a tiny portion of their results. The main theorem of Arnold-Vassiliev  describes the weak homotopy/homology types of the spaces $\mathcal F_{\leq k}$ for all $k \geq 2$. In particular, the homology of the space  $\mathcal F_{\leq 2}$ is isomorphic to the homology of $\Omega S^2$, the space of loops on a $2$-sphere (\cite{V})!  Arnold proved also that the fundamental group $\pi_1(\mathcal F_{\leq 2}) \approx \Z$ \cite{Ar}. 
\smallskip

For an even non-negative integer $d$, we  will also explore the subspaces $\mathcal F^d_{\leq 2} \subset \mathcal F_{\leq 2}$, formed by functions whose \emph{degree}---the sum of multiplicities of all its zeros---is even and does not exceed $d$.
\smallskip

Let $\hat v$ be a boundary convex traversing vector field on an annulus $A$. With the help of $\hat v$, we can introduce a product structure $A \approx S^1 \times [0, 1]$ so that the fibers of the projection $A \to S^1$ are the $\hat v$-trajectories. 

\begin{definition}\label{def1.6} Consider a collection $L$ of several smooth immersed loops in the annulus $A$ which intersect and self-intersect transversally and do not have triple intersections.  

We say that a boundary convex traversing vector field $\hat v$ is \emph{generic with respect to} $L$, if no $\hat v$-trajectory $\g$ contains more than one point of self-intersection from $L$ and no more than one point of simple tangency to $L$, but not both. 
 \hfill $\diamondsuit$
\end{definition}

For a given $L$, by standard techniques of the singularity theory,  we can find a perturbation of $\hat v$ within the space $\mathcal V_{\mathsf{trav}}(A)$ so that the perturbed field is generic with respect to $L$. \smallskip

Since an \emph{immersion} is a smooth map of manifolds, whose differential has the trivial kernel,  the immersions allow for a \emph{transfer} of a given vector field on the target manifold to a vector field in the source manifold. The transfer of a non-vanishing field is a non-vanishing field. \smallskip

All surfaces in this section are orientable. Note that any orientable surface $X$ admits an \emph{immersion} $\a: X \to A$ (or even in the plane $\R^2$) (see Fig. 8). We will use this fact to pull-back non-vanishing  fields on the target space $A$ to $X$. 

\begin{definition}\label{def1.7} Consider an \emph{immersion} $\a: X \to A$ of a  given compact orientable surface  $X$ into an annulus $A$, equipped with a traversal boundary convex (``radial") field $\hat v$. We call such $\a$ \emph{generic relative to} $\hat v$, if $\hat v$ is  generic with respect to the curves $\a(\d X)$ in the sense of Definition 9.1. 
\smallskip

Given a transversally generic field $v$ on a connected compact surface $X$, we call a map $\a: (X, v) \to (A, \hat v)$ a \emph{convex quasi-envelop of} $(X, v)$ 
 if there exists  an immersion $\a: X \to A$  which  is generic relative to the radial field $\hat v$ on $A$, and $v = \a^\ast(\hat v)$, the pull-back of $\hat v$. \hfill $\diamondsuit$
\end{definition}

\begin{figure}[ht]\label{fig1.8}
\centerline{\includegraphics[height=2.8in,width=3.4in]{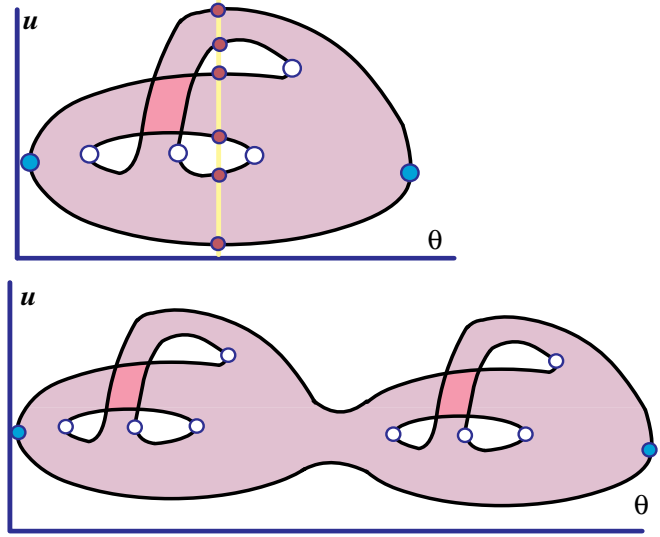}}
\bigskip
\caption{\small{A convex quasi-envelop $\a: X \to A$ of a traversally generic  field $\a^\ast(\d_u)$ on a punctured torus $X$ (on the top) and on a punctured surface $X$ of genus $2$ (on the bottom). In both examples, the cardinality of the $\theta$-fibers $\leq 6$.}}
\end{figure}




 


Given a boundary generic relative to $\hat v$ immersion $\a: X \to A$, the $\a$-pullback (transfer) of the field $\hat v$ defines a vector field $v \neq 0$ on $X$.  Since $\a$ is an immersion, evidently the pull-back $v$ is traversing on $X$. Moreover, $v$ is taversally generic in the sense of Definition \ref{def1.4}, since no $v$-trajectory $\g$ has more than one point of simple tangency to $\d X$. 

\begin{definition}\label{def1.8} Let $\a: X \to A$ be a \emph{regular embedding}  of a  given compact surface  $X$ into an annulus $A$, carrying a traversal boundary convex field $\hat v$. We denote by $v$ the pull-back of $\hat v$ under $\a$. If $\a$ is traversally generic relative to $\hat v$, then we say that  the pair $(A, \hat v)$ is a \emph{convex envelop} of $(X, v)$.

 \hfill $\diamondsuit$
\end{definition}

The existence of a convex envelop puts significant restrictions of the topology of $X$: such orientable surfaces $X$ do not have $1$-handles. In other words, they are disks with holes. 

\begin{lemma} If a compact connected surface $X$ with boundary has a pair of loops whose transversal intersection is a singleton, then no traversal flow on $X$ admits a convex envelop. In other words, if a connected surface $X$ with boundary has a handle, then no traversal flow on $X$ can be  convexly enveloped. 
\end{lemma}

\begin{proof} By Lemma 1.2, the space $\hat X$ of a convex envelop is either a disk or an annulus, both surfaces residing in the plane. No two loops in the plane intersect transversally at a singleton. Thus, for surfaces with a handle, no convex envelops exist.
\end{proof}

So the existence of a convex envelop severally restricts the topology of surface $X$. To incorporate surfaces with handles into our constructions, we have introduced the notion of a convex quasi-envelop (Definition \ref{def1.7}). \smallskip

Now we are in position to explore a connection between immersions $\a: (X, v) \subset (A, \hat v)$ of a given surface $X$ in the annulus $A$, such that $v = \a^\ast(\hat v)$ and $\hat v$ is generic with respect to $\a(\d X)$ on one hand, and loops in the functional spaces $\mathcal F_{\leq 2}$ on the other.\smallskip

Let $\a(\d X)^\times$ denote the set of self-intersections of  the curves forming the image $\a(\d X)$. Let $\a(\d X)^\circ$ denote the set $\a(\d X) \setminus \a(\d X)^\times$. \smallskip

With the pattern $\a(\d X)$ we associate an auxiliary smooth function $z_\a: A \to \R$, subject to the following properties:
\begin{itemize}
\item $z_\a^{-1}(0) = \a(\d X)$,
\item $0$ is the regular value of $z_\a$ at the points of $\a(\d X)^\circ$, 
\item in the vicinity of each point $a \in \a(\d X)^\times$, consider local coordinates $(x_1, x_2)$ such that $\{x_1 =0\}$ and $\{x_2 =0\}$ define the two intersecting branches of  $\a(\d X)$; then locally  $z_\a = c\cdot x_1x_2$, where the constant $c \neq 0$.
\item $z_\a =1$ in the vicinity of $\d A$,
\item the sign of $z_\a$ changes to opposite as a path crosses an arc from $\a(\d X)^\circ$ transversally\footnote{Thus the sign of $z_\a$ provides a ``checker board" coloring of the domains in $A \setminus \a(\d X)$.}.
\end{itemize}
\begin{eqnarray}\label{eq1.7}
\end{eqnarray}

Here we denote by $A^\circ$ the interior of the annulus $A$. Let $\phi: A^\circ \to \R$ be a smooth function so that $d\phi(\hat v) > 0$ in $A^\circ$ and $\phi(\hat\g \cap  A^\circ) = \R$ for all $\hat v$--trajectories $\hat\g$ in $A$. Then, with the help of $z_\a$ and $\phi$, we get a map $J_{z_\a}: \mathcal T(\hat v) \to \mathcal F_{\leq 2}$ whose target is the space of smooth functions $f: \R \to \R$ with no zeros of multiplicity $\geq  3$ and that are identically $1$ outside of a compact set in $\R$. We define the map $J_{z_\a}$ by the formula 
\begin{eqnarray}\label{eq1.8}
J_{z_\a}(\hat \g) = (z_\a|_{\hat \g}) \circ (\phi |_{\hat\g})^{-1},
\end{eqnarray}
where, abusing notations, $\hat \g$ stands for  both a $\hat v$-trajectory in $A$ and for the corresponding point in the trajectory space $\mathcal T(\hat v) \approx S^1$.

For a fixed $\a$, it is easy to check that the homotopy class $[J_{z_\a}]$ of $J_{z_\a}$ does not depend on the choice of the auxiliary function $z_\a$, subject to the five properties in (\ref{eq1.7}) (the space of such $z_\a$'s is convex and thus contractible). \smallskip

We pick a generator $\kappa \in \pi_1(\mathcal F_{\leq 2}) \approx \Z$ (see \cite{Ar}) and define the integer $J^\a$ by the formula $J^\a \cdot \kappa = [J_{z_\a}]$. As a result, any immersion $\a: X \to A$, which is generic with respect to $\hat v$,  produces a homotopy class $[J_{z_\a}] \in \pi_1(\mathcal F_{\leq 2})$ and an integer $J^\a$.

The isomorphism $\pi_1(\mathcal F_{\leq 2}) \approx \Z$ follows from the work of V. I. Arnold \cite{Ar} by a slight modification of his arguments, which we will describe next (see Theorem \ref{th1.7}). The main difference between our constructions and the ones from \cite{Ar} is that Arnold uses the critical loci of functions from $\mathcal F_{\leq 2}$, while we are using the zero loci.\smallskip

Generic loops in $\b: S^1 \to \mathcal F_{\leq 2}$ have an interpretation in terms of finite collections $C$ of smooth closed curves in the annulus $A$ with \emph{no inflection points with respect to  their tangent lines of the form} $\{\theta = const\}$ in the $(u, \theta)$-coordinates. We call such tangent lines $\theta$-\emph{vertical}. Furthermore, the generic homotopy between such loops $\b$ correspond to some \emph{cobordism} relation between the corresponding plane curves, the cobordism also avoids the $\theta$-vertical inflections.


First, let us spell out the genericity requirements on the collections $C$ of closed curves in the annulus $A$:
\begin{enumerate}
\item $C \subset A$ is a finite collectionof closed  smooth \emph{immersed} curves $\{C_j\}_j$,
\item the projections $\{\theta: C_j \to S^1\}_j$ have  Morse type singularities only\footnote{This excludes the $\theta$-vertical inflections.}, 
\item the self-intersections and mutual intersections of the curves  $\{C_j\}_j$ are transversal and no triple intersections are permited, 
\item at each double intersection, the two banches of $C$ are not parallel to the $u$-coordinate, 
\item  the $\theta$-images of the intersections and of the critical values of $\{\theta: C_j \to S^1\}_j$ are all distinct in $S^1$, 
\item the cardinality of each fiber of $\theta: C \to S^1$ does not exceed a given natural number $d$.
\end{enumerate}
\begin{eqnarray}\label{eq1.9}
\end{eqnarray}

\begin{definition}\label{def1.9}
Given two collections $C_0$ and $C_1$ of immersed closed curves as in (\ref{eq1.9}), we say that they are \emph{cobordant with no $\theta$-vertical inflections}, if there is a smooth 
function $F: A \times [0, 1] \to \R$  such that:
\begin{itemize}
\item $0$ is a regular value of $F$,
\item the restriction of the projection $T: A \times [0, 1] \to [0, 1]$ to the zero set $W =_{\mathsf{def}} F^{-1}(0)$ is a Morse function,
\item $C_0 = W \cap (A \times \{0\})$ and $C_1 = W \cap (A \times \{1\})$,
\item for each $t \in [0, 1]$, 
the section $C_t  =_{\mathsf{def}} W \cap (A \times \{t\})$ is such that $C_t$ has no $\theta$-horizontal inflections\footnote{Note that the second bullet excludes the  triple intersections of $C_t$.} 
\item for each $t \in [0, 1]$  the cardinality of the fibers of $\theta: C_t \to S^1$ does not exceed a given natural number $d$.
\hfill $\diamondsuit$
\end{itemize} 
\end{definition}

It is possible to verify that the cobordism with no $\theta$-vertical inflections is an \emph{equivalence relation} among collections of curves as in (\ref{eq1.9}). Indeed, if $C$ is cobordant to $C'$ with the help of $F$, and $C'$ to $C''$ with the help of $F'$, then there exists a piecewise smooth function $F \cup F': A \times [0, 2] \to \R$ whose restriction to $A \times [0, 1]$ is $F$ and to $A \times [1, 2]$ is a $(+1)$-shift of $F'$. Smoothing $F \cup F'$ along $A \times \{1\}$ in the normal direction  and scaling down the interval $[0, 2]$ to $[0,1]$, produces the desired function-cobordism  $F \ast F':  A \times [0, 1] \to \R$. 

So we can talk about \emph{the set of bordisms} $\mathbf B_{\mathsf{no\, \theta-inflect.}}$, based on collections of closed curves in the annulus with no $\theta$-vertical inflections.    
This set is a \emph{group}: the operation $C, C' \Rightarrow C \ast C'$ is defined by the union $\tilde C \cup \tilde C' \subset A$, where $\tilde C \subset S^1 \times (0, 0.5)$ and $\tilde C' \subset S^1 \times (0.5, 1)$ are the images of $C$ and $C'$, scaled down in the $u$-direction by the factor $0.5$ and placed in sub-annuli of $A = S^1 \times [0, 1]$. The role of $-C$ is played by the mirror image of $C$ with respect to a vertical (equivalently, horizontal) line, a fiber of $\theta: A \to S^1$.

Note that this operation $\ast$ may affect the maximal cardinalities $d$ and $d'$ of the fibers $\theta: C \to S^1$ and $\theta: C' \to S^1$ in a somewhat unpredictable way. In any case,  the fiber cardinality of $\theta: C\ast C' \to S^1$ has the upper boundary $d + d'$.\smallskip

The previous constructions deliver the following proposition, a slight modification of Theorem from \cite{Ar}.

\begin{theorem}\label{th1.7} The fundamental group $\pi_1(\mathcal F_{\leq 2})$ is isomorphic to the bordism group \hfill\break $\mathbf B_{\mathsf{no\, \theta-inflect.}}$, based on finite collections of immersed loops with no $\theta$-vertical inflections in the annulus $A$ and subject to the constraints (\ref{eq1.9}). The isomorphism is induced by the correspondence $$K: \{\b: S^1 \to \mathcal F_{\leq 2}\} \Rightarrow \{\b(\theta)^{-1}(0)\}_{\theta \in [0, 2\pi]} \subset A.$$ \hfill $\diamondsuit$
\end{theorem}

This theorem is a foundation of a \emph{graphic calculus} that converts homotopies of loops in the functional space $\mathcal F_{\leq 2}$ into cobordisms  of closed loop patterns in the annulus $A$ with no $\theta$-vertical inflections.

Figures 10 - 14 show an application of this calculus. They explain why any loop in $\mathcal F_{\leq 2}$ is homotopic to an integral multiple of a generator $\kappa \in \pi_1(\mathcal F_{\leq 2})$, represented by a model loop pattern $K \subset A$ as in Fig. 9, diagram (a) or (b). 
\smallskip

We orient the annulus $A = S^1 \times [0, 1]$ so that the the $\theta$-coordinate,  corresponding to $S^1$, is the first, and the $u$-coordinate, corresponding to $[0, 1]$, is the second. 

We fix an orientation of $X$, thus picking orientations for each component of $\d X$. Given an orientation-preserving immersion $\a: (X, v) \subset (A, \hat v)$ such that $\a(\d X)$ has  the properties as in (\ref{eq1.9}), we notice that the polarity of $a \in \d_2^+X(v)$ is $\oplus$ if and only if $\a_\ast(\nu_a)$, where $\nu_a$ is the inner normal to $\d X$ at $a$,  points in the direction of $\theta$. Otherwise, the polarity of $a$ is $\ominus$ (see Fig. 9).
\smallskip


\begin{theorem}\label{th1.8} Any orientation-preserving immersion $\a: (X, v) \subset (A, \hat v)$ such that $\hat v$ is generic with respect to $\a(\d X)$\footnote{for any convex quasi-envelop $\a$ of $(X, v)$}  produces a map $J_{z_\a} : S^1 \to \mathcal F_{\leq 2}$ (see (\ref{eq1.8})). Its homotopy class  $[J_{z_\a}] = J^\a \cdot \kappa$, where $\kappa$ denots a generator of $\pi_1( \mathcal F_{\leq 2}) \approx \Z$. 

The integer $J^\a$ can be computed by the formula:
\[J^\a =  \#\{\d_2^{+, \oplus}X(v)\} - \#\{\d_2^{+, \ominus}X(v)\}
\]
and thus does not depend on $\a$ (as long as the transfer $\a^\ast(\hat v) = v$).

Moreover, $|J^\a|  \leq c_2^+(v)$, the complexity of the $v$-flow.\smallskip



\end{theorem}

\begin{proof} Let $d =_{\mathsf{def}} \max_{\hat\g} \# \{\hat\g \cap \a(\d X)\}$ be the maximal cardinality of the intersections of the $\hat v$-trajectories $\hat\g$ with the loops' pattern $\a(\d X)$. Since $X$ bounds $\d X$, $d$ is even.

For any  $\hat v$-generic immersion $\a: X \subset A$, we pick an auxiliary function $z_\a : A \to \R$, adjusted to $\a$ as in (\ref{eq1.9}).  By the previous arguments, this choice  produces the loop  $J_{z_\a}: S^1 \to  \mathcal F^d_{\leq 2}$. Although the loop $J_{z_\a}$ is generated by an immersion $\a: X \to A$, in the process of deforming $J_{z_\a}$ by a cobordism $F: A \times [0, 1] \to \R$ with no $\theta$-vertical inflections as in Definition 9.4, we may destroy this connection with the original $\a$:  the new curve patterns $\{C_t\}_{t \in [0, 1]}$ in $A$ may not be produced by immersions $\{\a_t: X \to A\}_{t \in [0, 1]}$.\smallskip

Let us describe an algorithm (see Figures 10 - 14)  that reduces a given pattern $C_0 = J_{z_\a}^{-1}(0) \subset A$ to a pattern from \emph{the canonical set of patterns} $\{n\cdot K\}_{n \in \Z}$ (as in Fig. 9) by a cobordism $F: A \times [0, 1] \to \R$. We will perform a sequence of elementary surgeries on the set $C_0$, executed inside of the cylindrical shell $A \times [0, 1]$.  It is sufficient to construct a smooth surface $W \subset A \times [0, 1]$ as in Definition 9.4, for which $W \cap A \times \{0\} = C$ and $W \cap A \times \{1\} = J^\a \cdot K$; then one can define a function $F: A \times [0, 1] \to \R$, appropriately adjusted to $W$, so that $0$ is a regular value of $F$ and $F^{-1}(0) = W$.  

\begin{figure}[ht]\label{fig1.9}
\centerline{\includegraphics[height=2.6in,width=4.5in]{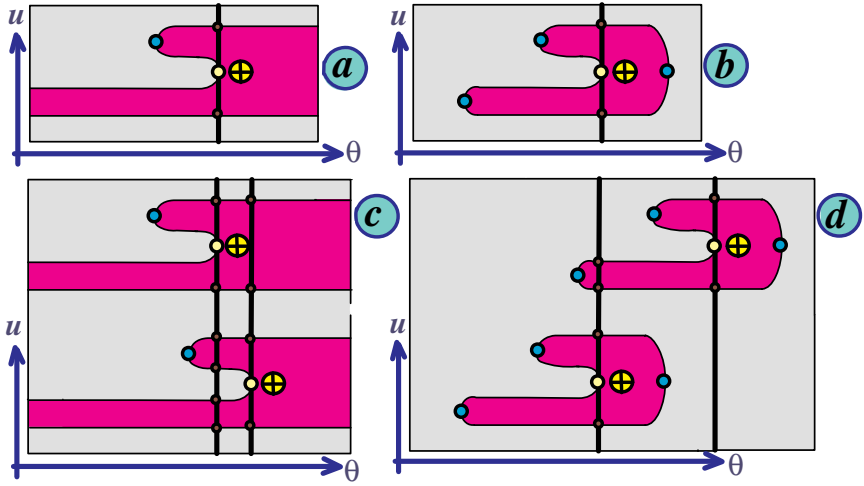}}
\bigskip
\caption{\small{Two equivalent representations of a generator $\kappa \in \pi_1(\mathcal F_{\leq 2})$ (diagrams (a) and (b)). Diagrams (c) and (d) portray $2\kappa$. Note the polariity $\oplus$ of the tangent $\hat v$-trajectories with the combinatorial pattern $(\dots 121 \dots)$. A mirrow image of these shapes with respect to a vertical line delivers $-\kappa$ and $-2\kappa$.}}
\end{figure}

As we modify the $t$-section $C_t \subset A \times \{t\}$, we keep track of the checker board polarities $+, -$, attached to the regions of $A \setminus C_t$; through the process, the polarity of the region adjacent to $\d A$ remains ``$+$".  Let us denote by $A_t^-$ the region of the negative polarity that is ``bounded" by the curve pattern $C_t \subset A \times \{t\}$. $A_t^+$ denotes the complementary set. Informally, the regions of polarity $+$ are the regions where the function $F$ from Definition \ref{def1.9} is non-negative.

With the help of this polarization $\{A_t^+, A_t^-\}$ of the annulus $A$, the points $a \in \d_2(C_t, \hat v)$, where $\hat v$-flow is tangent to $C_t$, acquire the polarization ``$+$" or ``$-$": if the germ of the trajectory $\hat\g_a$ is contained in $A_t^-$, then the polarity of $a$ is defined to be ``$+$", otherwise it is ``$-$". Moreover, if the inner normal $\nu_a$ to the region $A^-_t$ at $a$ has the same direction as the coordinate $\theta$ on $A$, then the second polarity of $a$ is defined to be ``$\oplus$", otherwise it is ``$\ominus$". As a result, we can talk about the four sets: $\d_2^{+, \oplus}(A_t^-, \hat v)$, $\d_2^{+, \ominus}(A_t^-, \hat v)$, $\d_2^{-, \oplus}(A_t^-, \hat v)$, 
$\d_2^{-, \ominus}(A_t^-, \hat v)$. We simplify the notations for these loci as: $\d_2^{+, \oplus}C_t, \; \d_2^{+, \ominus}C_t, \; \d_2^{-, \oplus}C_t, \; \d_2^{-, \ominus}C_t.$
\smallskip

\begin{figure}[ht]\label{fig1.10}
\centerline{\includegraphics[height=4in,width=4in]{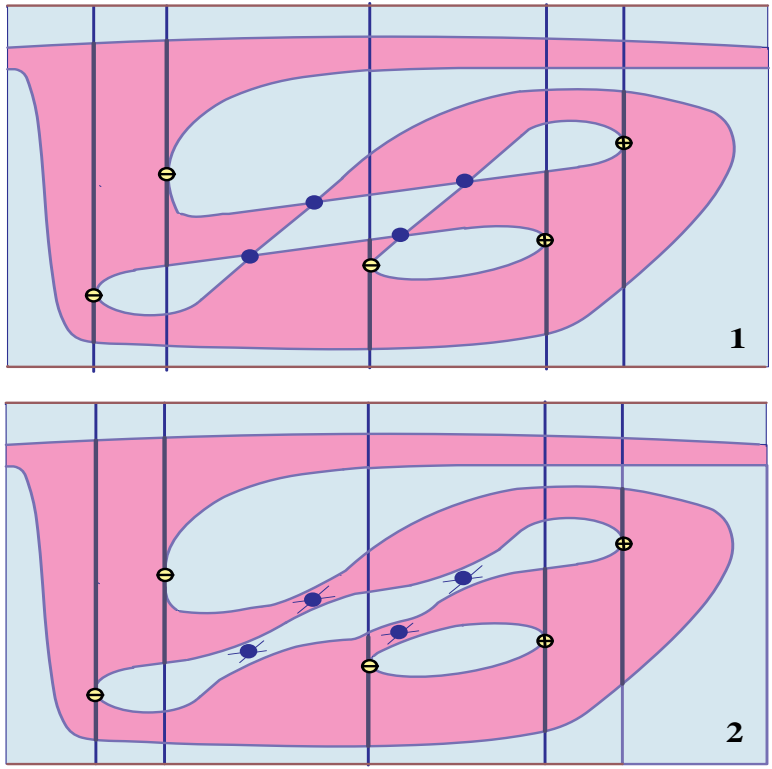}}
\bigskip
\caption{\small{An immersion $\a$ of a  surface $X$---the torus with two holes---in an annulus $A$ (represented as a rectangle with the two vertical sides to be glued).  The points whose $\a$-fiber has cardinality $2$ form a unshaded parallelogram (diagram 1). Eliminating crossings of $\a(\d X)$ by $1$-surgery (diagram 2).}}
\end{figure}

\begin{figure}[ht]\label{fig1.11}
\centerline{\includegraphics[height=4in,width=4in]{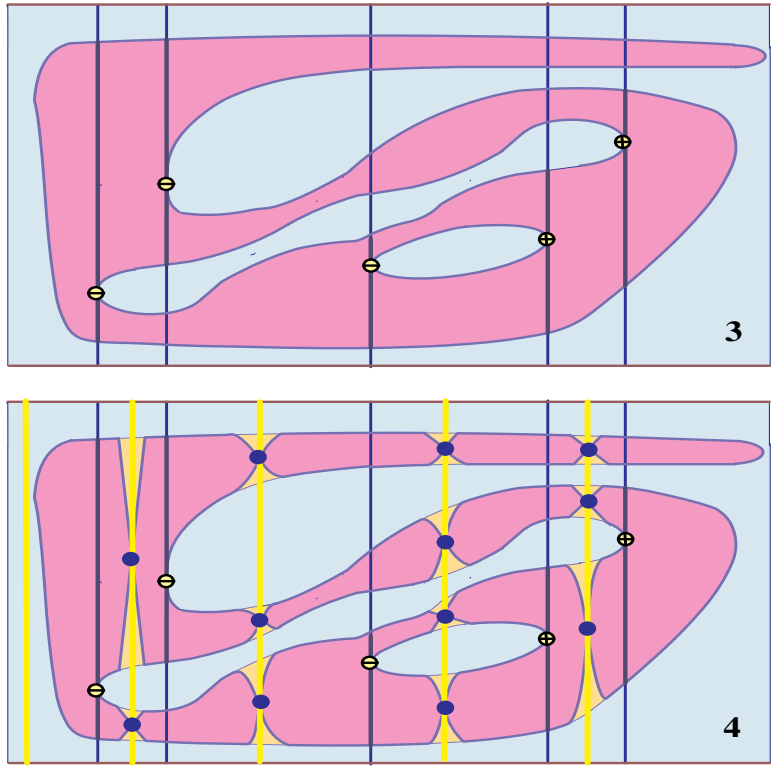}}
\bigskip
\caption{\small{$1$-surgery on $A_t^+$ places the curve pattern $C_t$ in a rectangle (diagram 3). In preparation for the next step, new crossings are introduced momentarily so that they reside in the $\theta$-fibers that separate the trajectories through the set $\d_2^+C_t$ (diagram 4).}}
\end{figure}

\begin{figure}[ht]\label{fig1.12}
\centerline{\includegraphics[height=4in,width=4in]{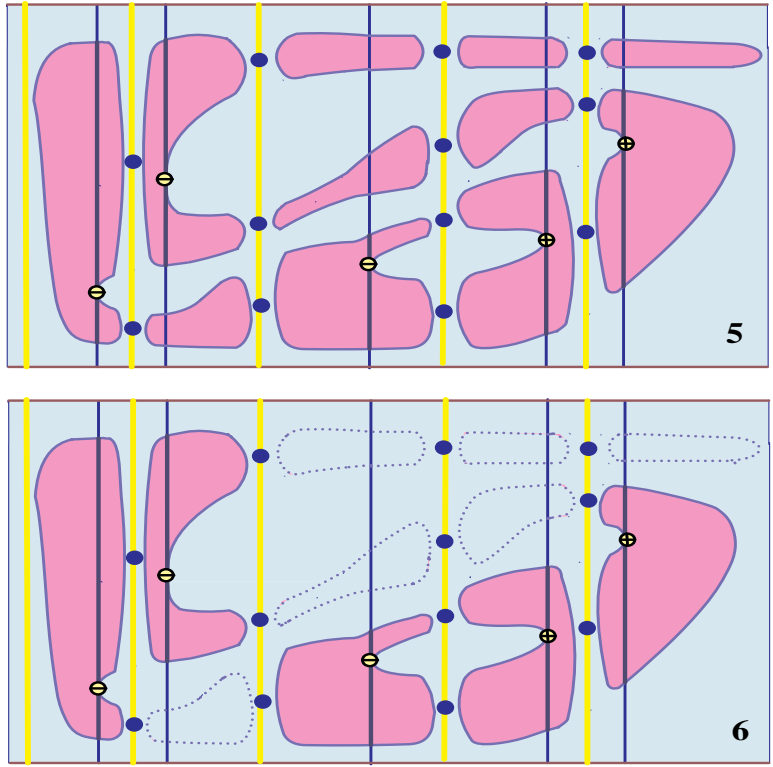}}
\bigskip
\caption{\small{Further $1$-surgery on $A_t^-$ breaks the curves from $C_t$ in groups, confined to $\theta$-vertical strips, so that each strip contains a single point from $\d_2^+C_t$ (diagram 5). Then the disks with no points from $\d_2^+C_t$ are eliminated by $2$-surgery on $A_t^+$ (diagram 6).}}
\end{figure}

\begin{figure}[ht]\label{fig1.13}
\centerline{\includegraphics[height=4in,width=4in]{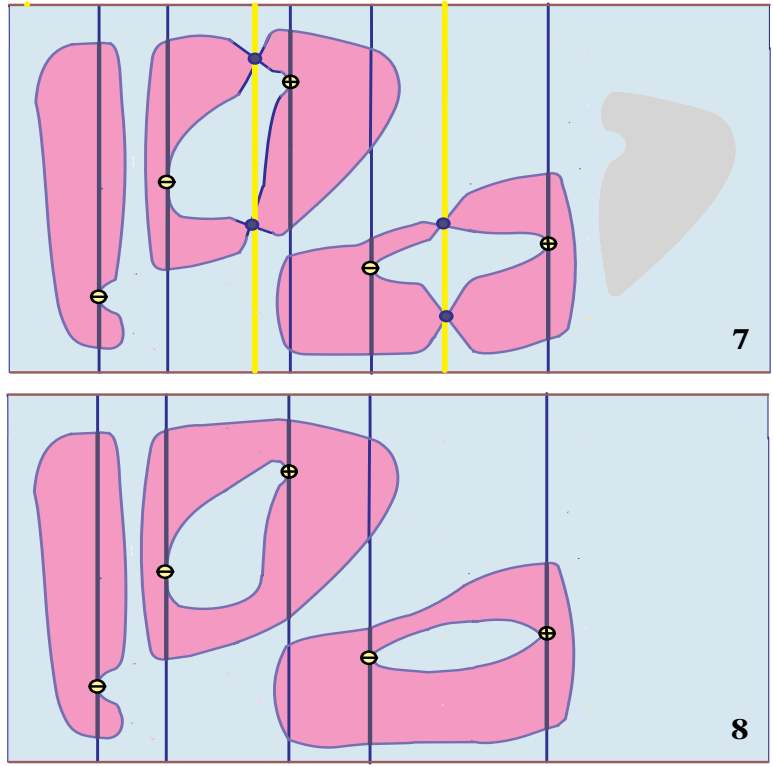}}
\bigskip
\caption{\small{In preparation for cancellation of trajectories of the opposite second polarity, the regions of $A_t^-$ are moved by an isotopy which preserves the fibers of $\theta: A \to S^1$ (the isotopy is not required by our $7$ step algorithm; it is applied only to decrease the number of figures), and the new crossings are created momentarily (diagram 7). Completing the $1$-surgery places each pair of trajectories of the opposite second polarity in an annulus, a portion of $A_t^-$. }}
\end{figure}

\begin{figure}[ht]\label{fig1.14}
\centerline{\includegraphics[height=4in,width=4in]{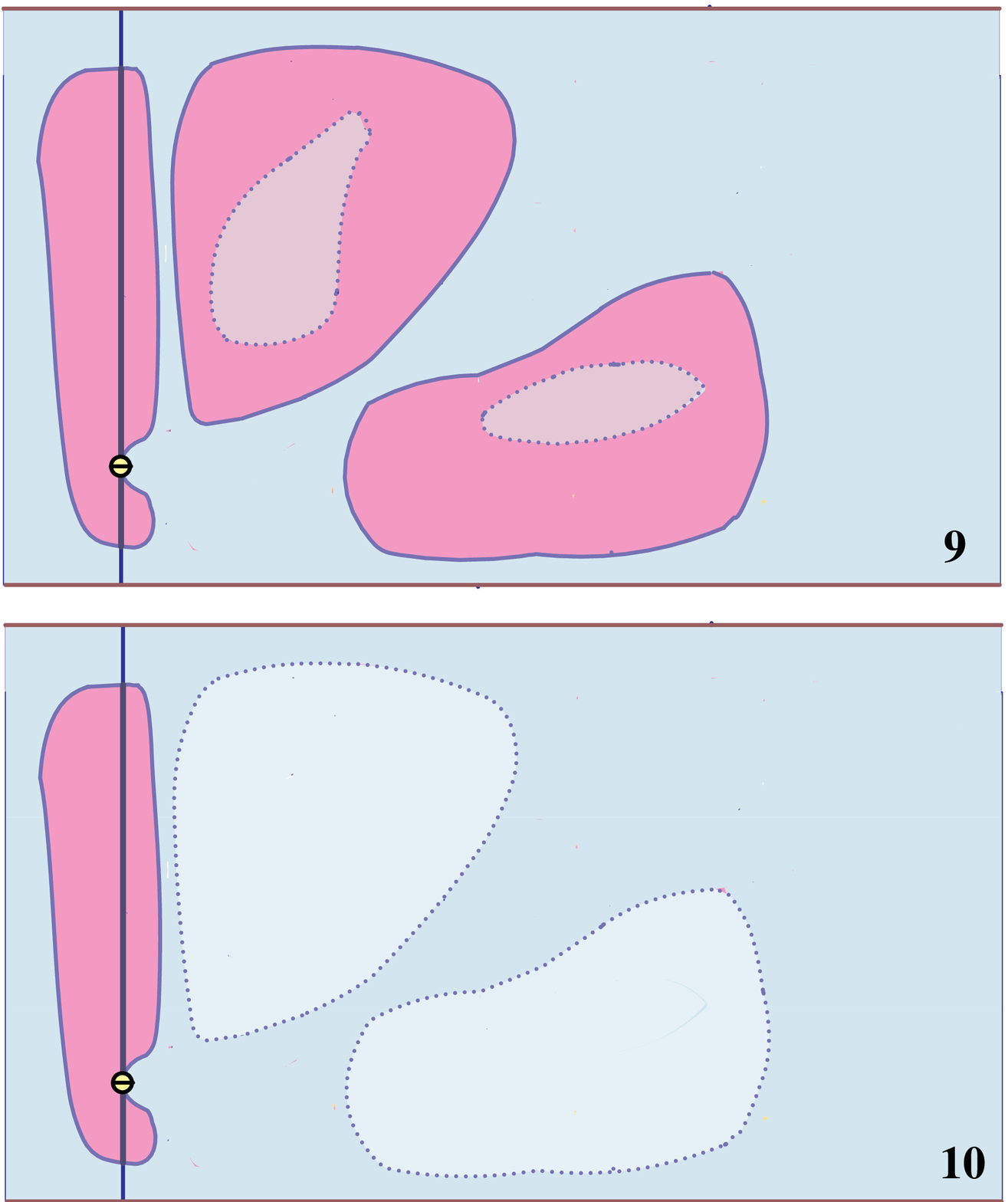}}
\bigskip
\caption{\small{Performing $2$-surgeries on the inner circles of the annuli in Fig. 13 converts $A_t^-$ into a disjoint union of disks, each disk containing a single trajectory of the same second polarity $\{\oplus, \ominus\}$ at most (diagram 10). Further $0$-surgery eliminates the disks without points of the first polarity $+$ (diagram 10). This leaves only the disks with a single point of the first polarity $+$ and a common second polarity ($\ominus$ in this example, so that $J^\a = -1$ for the $\a$ in Fig. 10).}}
\end{figure}

Let us describe an algorithm that constructs a cobordism with no $\theta$-vertical inflextions between a given loop pattern $\a(\d X)$ and a few copies of the canonical pattern as in Fig. 9.

{\bf (1)} At any stage of this construction, we can resolve each crossing $a \in(C_t)^\times$ in a preferred way. The two branches of the preferred resolution will be transversal to the $\theta$-fiber through $a$.  When the vector $\hat v(a)$ points inside $A_t^-$, the resolution will add a 1-handle to $A_t^+$, when the vector $\hat v(a)$ points inside $A_t^+$, the resolution will add a 1-handle to $A_t^-$. In any case, the sets $\d_2^{+, \oplus}C_t$, $\d_2^{+, \ominus}C_t$ are not affected. As a result of these resolutions, the new pattern $C'_t \subset A$ is a disjoint union of \emph{simple} smooth curves with no $\theta$-vertical inflections.  Moreover,  it shares with $C_0$ the same sets  $\d_2^{+, \oplus}\sim$, $\d_2^{+, \ominus}\sim$ \footnote{Later on, we may be forced to introduce momentarily new crossings for the exceptional sections $C_t$'s, which eventually will be eliminated.} (see Fig. 10).

{\bf (2)} Next we will pick and fix one regular value $\theta_\star \in S^1$ of the map $\theta: C_t \to S^1$. The intersection of $\theta^{-1} \cap A_t^+$ consists of several intervals. By performing $1$-surgery of $C_t$ along each  of these intervals we get a new loop pattern $C'_t$ which has an empty intersection with the $\theta$-fiber over the point $\theta_\star$. Moreover no $\theta$-vertical inflections were introduced in the process. The original loci $\d_2^{+, \oplus}\sim$, $\d_2^{+, \ominus}\sim$ are preserved (while the loci $\d_2^{-, \oplus}\sim$, $\d_2^{-, \ominus}\sim$ are changed). Therefore we may assume that $C_t = C'_t$ is contained in a rectangle $R \subset A$ and $C_t$ shares the numbers of points from the loci $\d_2^{+, \oplus}\sim$, $\d_2^{+, \ominus}\sim$ with the original $C_0$ (see Fig. 11, diagram 3). 

{\bf (3)} Consider the set $\Theta^+_t \subset S^1$ of critical values of $\theta: C_t \to S^1$ for the critical points from $\d_2^+C_t$.  We pick a regular value $\theta^\sharp_i$ in-between each pair of adjacent critical values $\theta_i, \theta_{i+1} \in \Theta^+_t $. Then we apply $1$-surgery on $A_t^+$ as in (2) to empty the region $A_t^-$ in the vicinity of the fiber $\theta^{-1}(\theta_i^\sharp)$ (see Fig. 11, diagram 4).

{\bf(4)} As a result of these surgeries, $A_t^-$ turns into a disjoint union of connected regions, each of which contains a single point of the set $\d_2^+(C_t)$ at most (see Fig. 12, 
diagram 5). 

{\bf (5)} By Lemma \ref{lem1.2}, any connected region of $A_t^- \subset R \subset A$ with no points from $\d_2^+\sim$ is a disk. It can be eliminated by a $0$-surgery (see Fig. 12, 
diagram 6). 

{\bf (6)} Pairs of  points $a \in  \d_2^{+, \oplus} \sim$ and $b \in \d_2^{+, \ominus}\sim$ can be cancelled via a surgery on their regions  $D_a, D_b \subset A_t^-$ (as shown in  Fig. 13, diagrams 7 and 8, and Fig. 14, diagrams 9 and 10).  This cancellation of pairs will be executed gradually and with some care. 

Any strip $S_i \subset A$, bounded by the vertical lines $\{\theta = \theta^\sharp_i \}$ and $\{\theta = \theta^\sharp_{i+1} \}$, contains a single region $D_a$ with $a \in \d_2^+C_t$. If the points of opposite second polarity ($\oplus, \ominus$) exist, then there are two \emph{adjacent} vertical strips $S_i, S_{i+1}$ such that  $a_i \in S_i \cap  \d_2^{+, \oplus} C_t$ and $b_i \in S_{i+1} \cap  \d_2^{+, \oplus} C_t$ have opposite second polarities.
We attach to $D_{a_i} \coprod D_{b_i}$ two $1$-handles to form an annulus $A_{a_ib_i}$ as in Figures 12 and 13. To complete the cancellation of opposite pairs, we perform $2$-surgery on the inner circles of the annuli $\{A_{a_ib_i}\}$. This converts the annuli into disks, residing in $A_t^-$. They can be eliminated by $0$-surgery along the outer circles of $\{A_{a_ib_i}\}$'s (see Fig. 14, diagram 10).

It may happen that the model domain $D_{a_i}$ as in Fig. 9, (b), and ``its mirror image" $D_{b_i}$ with respect to a vertical line $\{\theta = \theta_i^\sharp\}$ are positioned so that their horns are pointing in opposite directions. In such a case, they can be cancelled by a slightly different sequence of elementary surgeries (see \cite{Ar}). Alternatively, taking a trip around the annulus $A$, we will find a pair of adjacent strips such that  their domains $D_a$ and $D_b$ of opposite second polarity  can ``lock horns". For them, the previous recipe will apply. 

This cancellation procedure can be repeated by considering the remaining adjacent pairs of regions with the opposite second polarity untill no regions with the opposite second polarity are left. 

{\bf (7)} As a result of all these steps, $A^-_t$ is either empty, or a disjoint union of disks (as in Fig. 9), each of which contains a single point from $\d_2^+\sim$ (and tree points from $\d_2^-\sim$); \emph{the second polarities of such points are the same for all disks}. Thus we got an integral multiple of the basic pattern as in Fig. 9  and proved that $\pi_1(\mathcal F_{\leq 2})\approx \Z$.
\smallskip

Note that the original difference $\#\{\d_2^{+, \oplus}X(v)\} - \#\{\d_2^{+, \ominus}X(v)\}$ between the numbers of $\hat v$-trajectories with polarities $\oplus$ and $\ominus$ and of the combinatorial types $(\dots 121 \dots)$ is preserved under the modifications in (1)-(7). 

The original maximal cardinality $d$ of the $\theta$-fibers evidently does not increase under the steps  (1)-(7). 

Finally, we notice that 
$$c^+(v) =_{\mathsf{def}} \#\{\d_2^{+, \oplus}(v)\} +  \#\{\d_2^{+, \ominus}(v)\}$$ 
$$\geq \quad | \#\{\d_2^{+, \oplus}(v)\} - \#\{\d_2^{+, \ominus}(v)\}|. $$
\end{proof}

\noindent{\bf Remark 9.1.} 
It is interesting and somewhat surprising to notice that the invariant $J^\a =  \#\{\d_2^{+, \oplus}(v)\} - \#\{\d_2^{+, \ominus}(v)\}$ reflects more the topology of the field $v= \a^\ast(\hat v)$  than the topology of the surface $X$: in fact,  any integral value of $J^\a$ can be realized by a traversally generic  field $v$ on a \emph{disk} $D$ which even admits a convex envelop! A portion of  the boundary $\d D$ looks like a snake with respect to the field $\hat v$ of the envelop. For any $X$, the effect of deforming a portion of $\d X$ into a snake is equivalent to adding several times a spike (an edge and a pair of univalent and trivalent verticies) to the graph $\mathcal T(v)$. Evidently, these operations do not affect $H_1(\mathcal T(v); \Z) \approx H_1(X; \Z)$.

In contrast,  $\#\{\d_2^{+, \oplus}(v)\} + \#\{\d_2^{+, \ominus}(v)\} \geq 2|\chi(X)|$ has a topological significance for $X$. \smallskip

For example, for $\a$ as in Fig. 8, $J^\a =0$. If we subject $\a$ to an isotopy that introduces a snake-like pattern of Fig. 9, (a), then for the new immersion $\a'$, the invariant $J^{\a'} = 1$. \hfill $\diamondsuit$
\smallskip

\noindent{\bf Remark 9.2.} 
Consider a connected oriented surface $X$ with a connected boundary.  It is a boundary connected sum of a few copies of $T^\circ$, the torus with a hole. A punctured torus admits an immersion $\a: T^\circ \to A$ in the annulus so that the cardinality of the fibers of $\theta: \a(\d T^\circ) \to S^1$ does not exceed $6$ (see Fig. 8). Therefore,  any connected oriented surface $X$ with boundary admits an immersion $\a: X \to A$ with the property $\#\{ \theta^{-1}(\theta_\star) \cap \a(\d X)\} \leq 6$ for all $\theta_\star \in S^1$.
\hfill $\diamondsuit$

\smallskip

Let us glance at the implications of Theorems \ref{th1.7}  and \ref{th1.8} and give them a new, perhaps, more natural spin. \smallskip

The finite-dimensional space $\mathcal P^d_{\leq 2}$ of real monic polynomials of  an even degree $d$ and with no real roots of multiplicity $\geq 3$ is a natural ``approximation" of  the functional space $\mathcal F^d_{\leq 2}$. Of course, a polynomial from $\mathcal P^d_{\leq 2}$ is not a function from $\mathcal F_{\leq 2}$: it is not identically $1$ outside of a compact set. However, there is an embedding $\mathcal I_d:   \mathcal P^d_{\leq 2}  \to \mathcal F_{\leq 2}$ that, in the vicinity of $\pm \infty$, ``levels down to $1$"  any real polynomial $P$ of an even degree $d$. Its image belongs to the subspace $\mathcal F^d_{\leq 2}$. This embedding is described by an analytic formula (see \cite{V}) as follows. Fix an auxiliary smooth function $\chi: \R \to [0, 1]$ such that $\chi(u) = 0$ for $|u| \leq 1$,  $\chi(u) = 1$ for $|u| \geq 2$, and $\d\chi/\d u \neq 0$ for $1 < |u| < 2$. Let $\mu(P)$ denote the sum of absolute values of the coefficients of the monic  $P$. Then $$\mathcal I_d(P)(u) =_{\mathsf{def}} P(u) + (1 - P(u))\cdot \chi(u/\mu(p)).$$

In fact, the zeros of any polynomial $P$ are in 1-to-1 correspondence with the zeros of  the function $\mathcal I_d(P)$ and their multiplicities are preserved.
\smallskip

Consider the ``forbidden set" $\mathcal F^d_{\geq 3}  \subset \mathcal F^d$ of functions $f$ that have at least one zero of multiplicity $\geq 3$. Among them, the functions $f$ that have exactly one zero of multiplicity $3$ form an open and dense subset $(\mathcal F^d_{\geq 3})^\circ$. 

For each $f \in (\mathcal F^d_{\geq 3})^\circ$, let  $u^\star_f$ be the unique zero of multiplicity $3$. \smallskip

The set $\mathcal F^d_{\geq 3}$ has codimension $2$ in $\mathcal F^d$; so loops in $\mathcal F^d_{\leq 2}$ may be \emph{linked} with the locus  $\mathcal F^d_{\geq 3}$ in $\mathcal F^d$. 
Here is a model example of such a link (see  Fig. 15). 

\begin{figure}[ht]\label{fig1.16}
\centerline{\includegraphics[height=4in,width=4.3in]{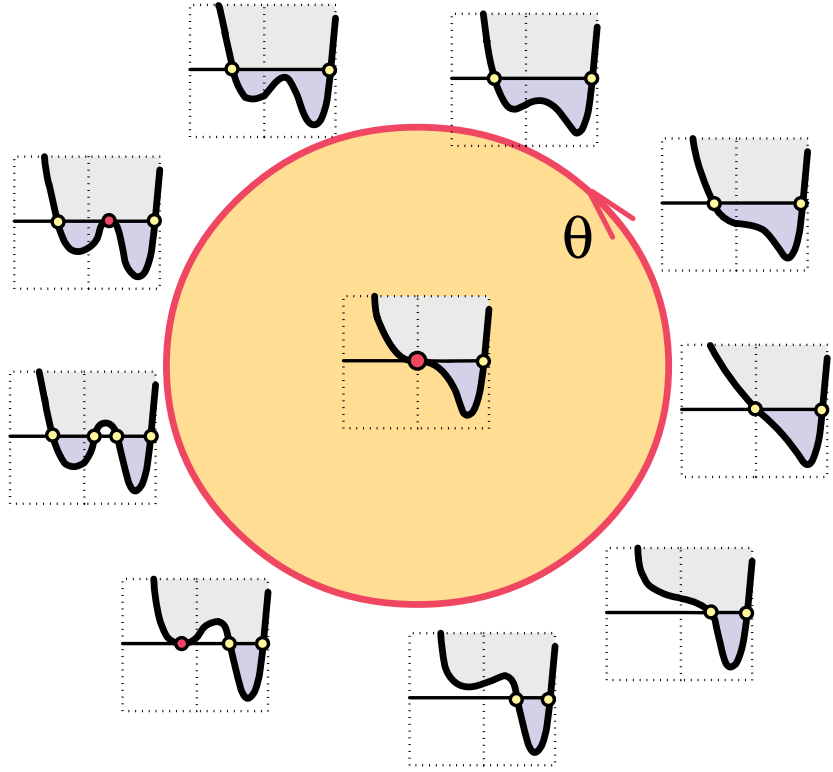}}
\bigskip
\caption{\small{A loop $\big\{P(\theta, u) =_{\mathsf{def}}(u-2)[u^3 + \cos(\theta) u + \sin(\theta)]\big\}_{\theta \in [0, 2\pi]}$ in the space $\mathcal P^4_{\leq 2}$, which represents a generator of $\pi_1(\mathcal P^4_{\leq 2})$.}}
\end{figure}

For any $u^\ast \in (-2, 2)$, consider a $\theta$-family of quartic $u$-polynomials $$\big\{P_{u^\star}(u, \theta) =_{\mathsf{def}} (u-2)\big[(u -u^\star)^3 + \cos(\theta) (u -u^\star) + \sin(\theta)\big]\big\}_{\theta \in [0, 2\pi]}.$$
Each $P_{u^\star}(u, \theta)$ belongs to the space $\mathcal P^4_{\leq 2}$. The $\mathcal I_4$-image of this $\theta$-family forms a loop $\{h_{u^\star}(u, \theta)\}_{\theta \in [0, 2\pi]}$ in $\mathcal F^4_{\leq 2}$. The loop bounds a $2$-disk $$D^2_{u^\star} =_{\mathsf{def}} \big\{\mathcal I_4\big((u-2)[(u -u^\star)^3 + x_1(u -u^\star) + x_0]\big)\big\}_{\{x_0^2 + x_1^2 \leq 1\}}$$ in $\mathcal F^4$, which hits  the subspace $\mathcal F^d_{\geq 3}$ at the singleton $\mathcal I_4\big((u-2)(u -u^\star)^3\big)$. 

Similarly, for any $f \in (\mathcal F^d_{\geq 3})^\circ$ and $d \geq 4$, the loop $$L_f =_{\mathsf{def}}\big\{f(u, \theta) = h_{u^\star_f}(u, \theta) \cdot \big[f(u)/(u- u^\star_f)^3\big]\big\}_{\theta \in [0, 2\pi]}$$
 resides in $\mathcal F^d_{\leq 2}$ and is linked with the component of  $(\mathcal F^d_{\geq 3})^\circ$ that contains $f$.

Since the correspondence $f \to u^\star_f$ is continuous for $f \in (\mathcal F^d_{\geq 3})^\circ$, the homotopy class  $[L_f] $ of the loop $L_f$ does not depend on the choice of $f$ within each component of $(\mathcal F^d_{\geq 3})^\circ$.  

In fact, by Theorem \ref{th1.8}, $[L_f]$ is a generator  $\kappa$ of $\pi_1(\mathcal F^d_{\leq 2})$. In particular, the $\mathcal I_4$-image of the loop $$ \big\{P(\theta, u) =_{\mathsf{def}}(u-2)[u^3 + \cos(\theta) u + \sin(\theta)]\big\}_{\theta \in [0, 2\pi]}$$  in $\mathcal P^4_{\leq 2}$ is a generator $\kappa$ of $\pi_1(\mathcal F_{\leq 2}) \approx \Z$. Its zero set in the annulus $A = S^1 \times [-3, 3]$ (equipped with the coordinates $(\theta, u)$) is a union of  two loops, similar to the ones shown in Fig. 9, (a). 
\smallskip

Theorem 11 from \cite{V}, makes an important for us claim:   the $\mathcal I_d$-induced map in homology $$(\mathcal I_d)_\ast: H_j(\mathcal P^d_{\leq 2};\, \Z) \to H_j(\mathcal F_{\leq 2};\, \Z)$$ is an isomorphism for all  $j \leq d/3$. In particular, $$(\mathcal I_4)_\ast: H_1(\mathcal P^4_{\leq 2};\, \Z) \to H_1(\mathcal F_{\leq 2};\, \Z) \approx \Z$$ is an isomorphism. Moreover, $$(\mathcal I_4)_\ast: \pi_1(\mathcal P^4_{\leq 2}) \to \pi_1(\mathcal F_{\leq 2}) \approx \Z$$
is an isomorphism as well \cite{Ar}. As we proceed, let us keep these facts in mind.
\smallskip

Of course it is much  easier  to visualize events in the $4$-dimensional space $\mathcal P^4_{\leq 2}$ than their analogues in the infinite-dimensional $\mathcal F^4_{\leq 2}$. This will be our next task. It will lead us to explore the beautiful stratified geometry of the Swallow Tail discriminant surface.

\begin{figure}[ht]\label{fig1.15}
\centerline{\includegraphics[height=3in,width=3in]{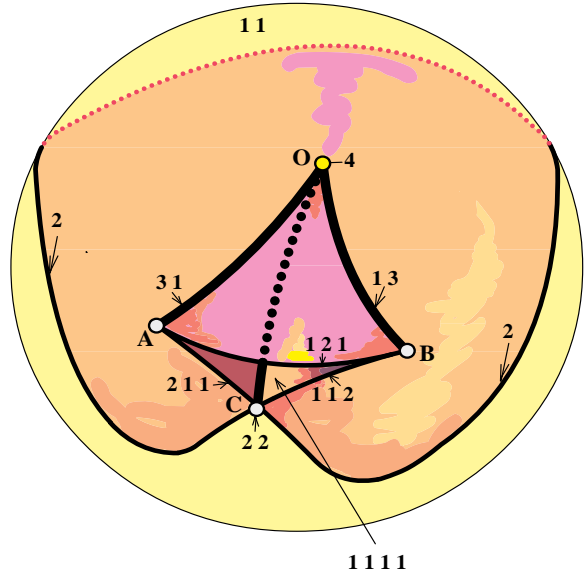}}
\bigskip
\caption{\small{The Swallow Tail Singularity is the critical locus of the Whitney projection of the hypersurface  $\{P(u, x_0, x_1, x_2) = u^4 + x_2 u^2 + x_1 u + x_0 = 0\}$ onto the space $\R_{\mathsf{coef}}^3$ with the coordinates $(x_0, x_1, x_2)$. The strata in  $\R_{\mathsf{coef}}^3$ are indexed by combinatorial types of real divisors of the polynomial $P(u, \sim)$: (1111), (11), $\emptyset$; (211), (121), (112); (22), (13), (31); (4). They divide the target space into three 3-cells, four 2-cells, three 1-cells, and one 0-cell.}}
\end{figure}

Consider the subspace $\tilde{\mathcal P}^4_{\leq 2} \subset \mathcal P^4_{\leq 2}$, formed by the monic depressed\footnote{that is, with the zero coefficient next to $u^3$} polynomials $P(u)$. Since $\tilde{\mathcal P}^4_{\leq 2}$ is a deformation retract of  $\mathcal P^4_{\leq 2}$ (\cite{Ar}), these two are homotopy equivalent. So it is  a bit easier to visualize the generator $\tilde\kappa \in \pi_1(\tilde{\mathcal P}^4_{\leq 2})$ (rather than $\kappa \in \pi_1(\mathcal P^4_{\leq 2})$), since $\tilde{\mathcal P}^4_{\leq 2}$ is a domain in $\R_{\mathsf{coef}}^3$, the space of the coefficients $x_0, x_1, x_2$. \smallskip

The polynomials with real roots of multiplicity $\geq 2$ form a singular surface $H \subset \R_{\mathsf{coef}}^3$, the famous Swallow Tail (see Fig. 16). The point-polynomial $u^4$ (the strongest singularity $O$ of $H$), together with the polynomials that have one root of multiplicity $3$, must be excluded from $\R_{\mathsf{coef}}^3$ to form $\tilde{\mathcal P}^4_{\leq 2}$. These excluded  polynomials form two branches of a curve $\mathcal C \subset H$, whose apex $O$ is the origin and whose branches extend to infinity. One branch $OB$ of $\mathcal C$ correspond to polynomials with the smaller simple root followed by the root of multiplicity $3$; the other branch $OA$  of $\mathcal C$ correspond to polynomials with the smaller root of multiplicity $3$, followed by the simple root. The curve $OC$, the self-intersection locus of $H$, represents polynomials with two distinct roots of multiplicity $2$. It belongs to the set $\tilde{\mathcal P}^4_{\leq 2}$ .

Thus $\pi_1(\tilde{\mathcal P}^4_{\leq 2}) = \pi_1(\R_{\mathsf{coef}}^3 \setminus \mathcal C) \approx \Z$. Therefore the generator $\tilde\kappa \in \pi_1(\tilde{\mathcal P}^4_{\leq 2})$ is represented by an oriented loop in $\R_{\mathsf{coef}}^3$ that winds once around the curve $\mathcal C$.  Such loop $\tilde\kappa$ must hit once the locus $H_{(121)} \subset H$, formed by polynomials whose real zeros conform to the pattern $(121)$: indeed, the curve $\mathcal C = OA \cup OB \cup O$ is the boundary of the surface $H_{(121)}$. 

Traversing $H_{(121)}$ from the chamber of polynomials with $4$ real roots to the chamber of polynomials with $2$ real roots picks a normal to $H_{(121)}$ orientation. Similar rule of orientation can be applied to the stratum $H_{(112)}$ bounded by the curves $OC \cup OB$,  the stratum $H_{(211)}$ bounded by the curves $OC \cup OA$, and the stratum $H_{(2)}$ that separates the chamber with two real roots from the chamber with no real roots at all. \smallskip

For any smooth loop $\b: S^1 \to \tilde{\mathcal P}^4_{\leq 2}$ which is in general position to $H$ and its strata, consider the zero set $$\D(\b) =_{\mathsf{def}}  \{(\theta, u)|\;  \b(\theta)(u) = 0\}\subset S^1 \times \R.$$
Thanks to the very definition of the target space  $\tilde{\mathcal P}^4_{\leq 2}$, the set $\D(\b)$ is a collection of closed curves with transversal self-intersections, no triple intersections, no self-tangencies, and no $\theta$-vertical inflections.  The cardinality of the fiber of $\theta: \D(\b) \to S^1$ does not exceed $4$.\smallskip

Now the linking number $J^\b =_{\mathsf{def}} \mathsf{lk}(\b(S^1), \mathcal C)$ equals  the algebraic intersection of the loop $\b(S^1)$ with the surface $H_{(121)}$ that bounds $\mathcal C$. Each time  the loop $\b(S^1)$ intersects the surface $H_{(121)}$ transversally in the direction of the positive normal, a point from $\d_2^{+, \oplus}\D(\b)$ is generated, and each time $\b(S^1)$ intersects $H_{(121)}$ is  in the direction of the negative normal, a point from $\d_2^{+, \ominus}\D(\b)$ is generated.  In particular,  the generator $\kappa: S^1 \to  \tilde{\mathcal P}^4_{\leq 2}$ of $\pi_1(\tilde{\mathcal P}^4_{\leq 2})$ has the property  $\mathsf{lk}(\b(S^1), \mathcal C) = 1$. Therefore we get $$\mathsf{lk}(\b(S^1), \mathcal C) = \b(S^1) \circ H_{(121)}  = \#\{\d_2^{+, \oplus}\D(\b)\} - \#\{\d_2^{+, \ominus}\D(\b)\},$$ the baby model of the formula from Theorem \ref{th1.8}. This number is an invariant of the homotopy class of the loop $\b$. \smallskip

Similarly, each time  the loop $\b(S^1)$ intersects the surface $H_{(2)}$ transversally in the direction of the positive normal, a point from $\d_2^{-, \oplus}\D(\b)$ is generated, and each time $\b(S^1)$ intersects $H_{(2)}$ is  in the direction of the negative normal, a point from $\d_2^{-, \ominus}\D(\b)$ is generated. \smallskip

Note that perhaps not any set $\D(\b)$ is the image $\a(\d X)$ of an immersion $\a: X \to S^1 \times \R$ for some orientable surface $X$. But the generator $\kappa$ in  Fig. 9,
(a), is.  Also, if we insist that the cardinality of the $\theta$-fibers $\leq 4$, we cannot accommodate surfaces $X$ with handles. According to Remark 9.2, to accommodate them, we need to deal with polynomials/functions of degree $6$ at least. 
\smallskip

Let us describe briefly how this ``degree $6$ polynomial model" works. We will see that the increasingly complex combinatorics of tangency begins to play a significant role. 

To simplify the notations, we identify $\mathcal  P^6$ with its image $\mathcal I_6(\mathcal  P^6) \subset \mathcal F^6$ and the cylinder $S^1 \times \R$ with the interior of the annulus $A$. 

The combinatorial patterns $\omega$ of real divisors of monic degree $6$ real polynomials are numerous: 
\begin{itemize}
\item (111111), (1111), (11), $\emptyset$;
\item (21111), (12111), (11211), (11121), (11112), (211), (121), (112), (2); 
\item (2211), (1221), (1122), (2121), (2112), (1212), \,  (22);
\item (3111), (1311), (1131), (1113), (31), (13); 
\item etc.
\end{itemize}

We denote by $\mathcal P^6_\omega$ the set of real monic polynomials whose real divisors conform to the combinatorial pattern $\omega = (\omega_1, \omega_2, \dots , \omega_s)$, where $\{\omega_i\}_{1 \leq i \leq s}$ are natural numbers. We denote by $|\omega|$ the $l_1$-norm of the vector $\omega$. Evidently, $|\omega| \leq 6$.

The sets  $\{\mathcal P^6_\omega\}_\omega$ form a partition of the space $\mathcal P^6$. In fact, each $\mathcal P^6_\omega$ is homeomorphic to an open ball of dimension $6 - |\omega|'$, where $|\omega|' = |\omega | - s$ (\cite{K3}). The closure $\bar{\mathcal P}^6_\omega$ of $\mathcal P^6_\omega$ in $\R^6_{\mathsf{coef}}$ is an affine semi-algebraic variety. By resolving $\bar{\mathcal P}^6_\omega$ appropriately,  one can show that  the partition $\{\mathcal P^6_\omega\}_\omega$ defines a structure of a $CW$-complex on $\mathcal P^6$, or rather, on its one-point compactification  (\cite{K3}). So we may think of $\mathcal P^6_\omega$'s as being ``cells" (although $\bar{\mathcal P}^6_\omega$ may not be homeomorphic to an infinite cone over a closed ball).

Let $H \subset \mathcal P^6$ denote the set of monic polynomials with multiple roots, the $5$-dimensional discriminant variety. The first bullet lists the four $6$-dimensional chambers-cells in which $H$ divides $\mathcal P^6$. The second bullet lists all $5$-dimensional strata in which $H$ is divided by the strata of dimension $4$. The first three bullets list the monic polynomials that form the space $\mathcal P^6_{\leq 2}$. The third and the fourth bullets list the $4$-dimensional cells-strata. The forbidden locus $\mathcal P^6_{\geq 3}$ is the union of strata, labeled by the combinatorial types in the fourth bullet and on. Then $\mathcal P^6_{\geq 3}$ is the closure of the set $$\mathcal P^6_{(3111)} \cup \mathcal P^6_{(1311)} \cup \mathcal P^6_{(1131)} \cup  \mathcal P^6_{(1113)} \cup \mathcal P^6_{(31)} \cup \mathcal P^6_{(13)}.$$  

We can orient each cell $\mathcal P^6_\omega$ so that  
$$\d \bar{\mathcal P}^6_{(12111)} =  \bar{\mathcal P}^6_{(3111)} - \bar{\mathcal P}^6_{(1311)} + \bar{\mathcal P}^6_{(1221)} -\bar{\mathcal P}^6_{(1212)},$$
$$\d \bar{\mathcal P}^6_{(11121)} =  \bar{\mathcal P}^6_{(1131)} - \bar{\mathcal P}^6_{(1113)} + \bar{\mathcal P}^6_{(2121)} - \bar{\mathcal P}^6_{(1221)},$$
$$\d \bar{\mathcal P}^6_{(121)} =  \bar{\mathcal P}^6_{(31)} -\bar{\mathcal P}^6_{(13)} + \bar{\mathcal P}^6_{(2121)} - \bar{\mathcal P}^6_{(1212)}.$$
The operator $\d$ in these formulas should be understood in the spirit of algebraic topology as the boundary operator on cellular chains (and not as a topological boundary of the appropriate sets) \cite{K6}.

Adding the three formulas above, we get that the forbidden set, viewed as a $4$-chain, is an algebraic boundary of a $5$-chain:

\[
\mathcal P^6_{\geq 3} =_{\mathsf{def}}\, \bar{\mathcal P}^6_{(3111)} - \bar{\mathcal P}^6_{(1311)}  + \bar{\mathcal P}^6_{(1131)} - \bar{\mathcal P}^6_{(1113)} +  \bar{\mathcal P}^6_{(31)} - \bar{\mathcal P}^6_{(13)} 
\] 
\[
= \d \big(\bar{\mathcal P}^6_{(12111)} + \bar{\mathcal P}^6_{(11121)} + \bar{\mathcal P}^6_{(121)}\big).
\]

\smallskip

Now consider a smooth loop $\b: S^1 \to \mathcal P^6_{\leq 2}$. By a small perturbation we may assume that $\b(S^1)$ is transversal to the hypersurfaces  $\mathcal P^6_{(12111)},  \mathcal P^6_{(11121)},  \mathcal P^6_{(121)}$ that bound the cycle $\bar{\mathcal P}^6_{\geq 3}$. Therefore 

$$\mathsf{lk}\big(\b(S^1), \mathcal P^6_{\geq 3}\big) = \b(S^1) \circ \big(\mathcal P^6_{(12111)} \cup \mathcal P^6_{(11121)} \cup  \mathcal P^6_{(121)}\big).$$

Again, we form the set $$\D(\b) =_{\mathsf{def}}  \{(\theta, u)|\;  \b(\theta)(u) = 0\}\subset S^1 \times \R.$$

Then $$\b(S^1) \circ \big(\mathcal P^6_{(12111)} \cup \mathcal P^6_{(11121)} \cup  \mathcal P^6_{(121)}\big) = \#\{\d_2^{+, \oplus}\D(\b)\} - \#\{\d_2^{+, \ominus}\D(\b)\}.$$\footnote{Note that the points of $\b(S^1) \circ \mathcal P^6_{(11211)}$ belong to the locus  $\d_2^{-, \sim}\D(\b)$.}
Therefore we get $$\mathsf{lk}\big(\b(S^1), \mathcal P^6_{\geq 3}\big) = \#\{\d_2^{+, \oplus}\D(\b)\} - \#\{\d_2^{+, \ominus}\D(\b)\},$$
a version of the formula from Theorem \ref{th1.8}, being applied to the loop $\b = J_{z_\a}$. The loop is produced by a generic with respect to $\hat v$ immersion $\a: X \to A$, such that the cardinality of the fibers of $\theta: \a(\d X) \to S^1$ does not exceed $6$, and by an auxiliary  function $z_{\a(\d X)}$.\smallskip

These considerations are not restricted to polynomials/functions of degree $6$: they apply to any even degree $d$. The application requires a deeper dive into the combinatorics of real polynomial divisors and their modifications, but  the spirit is captured by the arguments  that deal with degree $6$ (see [K3]).

\smallskip

Any $\hat v$-generic immersion $\a: X \to A$ also produces a well-defined element $[K_\a]$ in the set of \emph{homotopy classes} $[\mathcal T(v), \mathcal F_{\leq 2}]$ of maps from the trajectory graph $\mathcal T(v)$ to the functional space $\mathcal F_{\leq 2}$. Its construction is similar to the one of $J_{z(\a)}$. Consider the $\hat v$-generated obvious map $Q_\a: \mathcal T(v) \to \mathcal T(\hat v) \approx S^1$ (each $v$-trajectory is contained in the unique $\hat v$-trajectory).  Put  $K_\a =_{\mathsf{def}} J_{z(\a)}  \circ Q_\a$.
\smallskip

\noindent {\bf Remark 9.3.} Note that, for some immersions $\a : X \to A$, the invariant $J^\a$ may be different from $0$, but $[K_\a]$ may be trivial. For example, this is the case when $X$ is a disk with a snake-like boundary $\a(\d X)$ with respect to $\hat v$. However, there exist immersions $\a$ with a nontrivial $[K_\a]$. For example, such is the immersion in Fig. 10, (1). At the same time, for $\a$ in Fig. 11, (3), $[K_\a]$ is trivial.  \hfill $\diamondsuit$
\smallskip

Since $\pi_1(\mathcal F_{\leq 2}) \approx \Z$, it follows that $H_1(\mathcal F_{\leq 2}; \Z) \approx \Z$. In turn, this implies that the $1$-dimensional cohomology $H^1(\mathcal F_{\leq 2}; \Z) \approx \Z$. 

Thus $K_\a$ induces a map $$K_\a^\ast: H^1(\mathcal F_{\leq 2}; \Z) \to H^1(\mathcal T(v); \Z) \approx H^1(X; \Z).$$
In particular, we get an element $K_\a^\ast(\kappa^\ast) \in H^1(X; \Z)$, where $\kappa^\ast$ is a generator of $H^1(\mathcal F_{\leq 2}; \Z) \approx \Z$. This cohomology class $K_\a^\ast(\kappa^\ast)$ is a characteristic class of the given $\hat v$-generic immersion $\a$. 

Theorem \ref{th1.7} implies that if two $\hat v$-generic immersions $\a, \a_1: X \to A$ are such that the pull-backs $\a^\ast(\hat v) = \a^\ast_1(\hat v) = v$, then $K_\a^\ast(\kappa^\ast) = K_{\a_1}^\ast(\kappa^\ast)$. So the cohomology class $K_\a^\ast(\kappa^\ast)$ is, in fact, a characteristic class of $v$. It is desirable to be able to reach this conclusion without relying on the cobordisms of curves' patterns in $A$ with no $\theta$-vertical inflections. 

Based on the partial evidence, provided by the two polynomial models $\mathcal P^4_{\leq 2}$ and $\mathcal P^6_{\leq 2}$, we may conjecture that the value of  $K_\a^\ast(\kappa^\ast)$  on any loop ($1$-cycle) $\delta: S^1 \to X$ equals to the linking number $\mathsf{lk}(K_\a(\delta), \mathcal F_{\geq 3})$. The validation of this conjecture requires to extend our analysis of the stratified geometry of $\mathcal P^6_{\leq  2}$ to $\mathcal P^d_{\leq  2}$ and to show that $\mathcal I_d: \mathcal P^d_{\leq 2} \to \mathcal F^d_{\leq 2}$ is a weak homotopy equivalence for all even $d$. Both steps are realizable with the techniques developed in \cite{K3} and in \cite{K5}. \smallskip

In dimensions higher than two, similar considerations apply to produce characteristic classes of traversally generic flows. They are based on computations of homology of spaces of real monic polynomials with \emph{restricted combinatorics} of their real divisors. It turns out  that the topology of high-dimensional convex envelops is as intricate as the homotopy groups of spheres \cite{K6}. 
\bigskip

Our investigation of vector flows in Flatland reached its conclusion. To find out how things flow in other  lands---``the romances of many dimensions"---(\cite{Ab}), the reader could consult with the references below.
\bigskip


\end{document}